\documentclass[12pt]{article}
\usepackage[utf8]{inputenc}
\usepackage{listings}
\usepackage[document]{ragged2e}
\usepackage[papersize={21.59cm,27.94cm},tmargin=2.54cm, lmargin=3cm,rmargin=3cm,bmargin=2.54cm]{geometry}
\usepackage{enumerate}
\usepackage{tikz-cd}
\usepackage{pb-diagram}
\usepackage{dcolumn}
\usepackage{mathrsfs}
\usepackage{hyperref}
\usepackage{epsfig,graphicx,psfrag,dsfont}
\usepackage[font=small,labelfont=bf,textfont=it]{caption}
\usepackage{fancybox}
\usepackage{graphicx}
\usepackage{amsmath,amsfonts,amssymb,amsthm,epsfig,epstopdf, hyperref,array}
\usepackage{mathrsfs}
\usepackage[retainorgcmds]{IEEEtrantools}
\usepackage{makeidx,epsfig,lscape}
\usepackage{xcolor,pict2e}
\usepackage{upgreek}
\usepackage{tikz-cd}
\usepackage[breakable]{tcolorbox}
\usepackage[all]{xy}	
\usepackage{parskip} 
\usepackage{fancyvrb} 

\definecolor{urlcolor}{rgb}{0,.145,.698}
\definecolor{linkcolor}{rgb}{.71,0.21,0.01}
\definecolor{citecolor}{rgb}{.12,.54,.11}
\definecolor{ansi-black}{HTML}{3E424D}
\definecolor{ansi-black-intense}{HTML}{282C36}
\definecolor{ansi-red}{HTML}{E75C58}
\definecolor{ansi-red-intense}{HTML}{B22B31}
\definecolor{ansi-green}{HTML}{00A250}
\definecolor{ansi-green-intense}{HTML}{007427}
\definecolor{ansi-yellow}{HTML}{DDB62B}
\definecolor{ansi-yellow-intense}{HTML}{B27D12}
\definecolor{ansi-blue}{HTML}{208FFB}
\definecolor{ansi-blue-intense}{HTML}{0065CA}
\definecolor{ansi-magenta}{HTML}{D160C4}
\definecolor{ansi-magenta-intense}{HTML}{A03196}
\definecolor{ansi-cyan}{HTML}{60C6C8}
\definecolor{ansi-cyan-intense}{HTML}{258F8F}
\definecolor{ansi-white}{HTML}{C5C1B4}
\definecolor{ansi-white-intense}{HTML}{A1A6B2}
\definecolor{ansi-default-inverse-fg}{HTML}{FFFFFF}
\definecolor{ansi-default-inverse-bg}{HTML}{000000}

\definecolor{outerrorbackground}{HTML}{FFDFDF}

\DefineVerbatimEnvironment{Highlighting}{Verbatim}{commandchars=\\\{\}}

\makeatletter
\def\PY@reset{\let\PY@it=\relax \let\PY@bf=\relax%
	\let\PY@ul=\relax \let\PY@tc=\relax%
	\let\PY@bc=\relax \let\PY@ff=\relax}
\def\PY@tok#1{\csname PY@tok@#1\endcsname}
\def\PY@toks#1+{\ifx\relax#1\empty\else%
	\PY@tok{#1}\expandafter\PY@toks\fi}
\def\PY@do#1{\PY@bc{\PY@tc{\PY@ul{%
				\PY@it{\PY@bf{\PY@ff{#1}}}}}}}
\def\PY#1#2{\PY@reset\PY@toks#1+\relax+\PY@do{#2}}

\@namedef{PY@tok@w}{\def\PY@tc##1{\textcolor[rgb]{0.73,0.73,0.73}{##1}}}
\@namedef{PY@tok@c}{\let\PY@it=\textit\def\PY@tc##1{\textcolor[rgb]{0.24,0.48,0.48}{##1}}}
\@namedef{PY@tok@cp}{\def\PY@tc##1{\textcolor[rgb]{0.61,0.40,0.00}{##1}}}
\@namedef{PY@tok@k}{\let\PY@bf=\textbf\def\PY@tc##1{\textcolor[rgb]{0.00,0.50,0.00}{##1}}}
\@namedef{PY@tok@kp}{\def\PY@tc##1{\textcolor[rgb]{0.00,0.50,0.00}{##1}}}
\@namedef{PY@tok@kt}{\def\PY@tc##1{\textcolor[rgb]{0.69,0.00,0.25}{##1}}}
\@namedef{PY@tok@o}{\def\PY@tc##1{\textcolor[rgb]{0.40,0.40,0.40}{##1}}}
\@namedef{PY@tok@ow}{\let\PY@bf=\textbf\def\PY@tc##1{\textcolor[rgb]{0.67,0.13,1.00}{##1}}}
\@namedef{PY@tok@nb}{\def\PY@tc##1{\textcolor[rgb]{0.00,0.50,0.00}{##1}}}
\@namedef{PY@tok@nf}{\def\PY@tc##1{\textcolor[rgb]{0.00,0.00,1.00}{##1}}}
\@namedef{PY@tok@nc}{\let\PY@bf=\textbf\def\PY@tc##1{\textcolor[rgb]{0.00,0.00,1.00}{##1}}}
\@namedef{PY@tok@nn}{\let\PY@bf=\textbf\def\PY@tc##1{\textcolor[rgb]{0.00,0.00,1.00}{##1}}}
\@namedef{PY@tok@ne}{\let\PY@bf=\textbf\def\PY@tc##1{\textcolor[rgb]{0.80,0.25,0.22}{##1}}}
\@namedef{PY@tok@nv}{\def\PY@tc##1{\textcolor[rgb]{0.10,0.09,0.49}{##1}}}
\@namedef{PY@tok@no}{\def\PY@tc##1{\textcolor[rgb]{0.53,0.00,0.00}{##1}}}
\@namedef{PY@tok@nl}{\def\PY@tc##1{\textcolor[rgb]{0.46,0.46,0.00}{##1}}}
\@namedef{PY@tok@ni}{\let\PY@bf=\textbf\def\PY@tc##1{\textcolor[rgb]{0.44,0.44,0.44}{##1}}}
\@namedef{PY@tok@na}{\def\PY@tc##1{\textcolor[rgb]{0.41,0.47,0.13}{##1}}}
\@namedef{PY@tok@nt}{\let\PY@bf=\textbf\def\PY@tc##1{\textcolor[rgb]{0.00,0.50,0.00}{##1}}}
\@namedef{PY@tok@nd}{\def\PY@tc##1{\textcolor[rgb]{0.67,0.13,1.00}{##1}}}
\@namedef{PY@tok@s}{\def\PY@tc##1{\textcolor[rgb]{0.73,0.13,0.13}{##1}}}
\@namedef{PY@tok@sd}{\let\PY@it=\textit\def\PY@tc##1{\textcolor[rgb]{0.73,0.13,0.13}{##1}}}
\@namedef{PY@tok@si}{\let\PY@bf=\textbf\def\PY@tc##1{\textcolor[rgb]{0.64,0.35,0.47}{##1}}}
\@namedef{PY@tok@se}{\let\PY@bf=\textbf\def\PY@tc##1{\textcolor[rgb]{0.67,0.36,0.12}{##1}}}
\@namedef{PY@tok@sr}{\def\PY@tc##1{\textcolor[rgb]{0.64,0.35,0.47}{##1}}}
\@namedef{PY@tok@ss}{\def\PY@tc##1{\textcolor[rgb]{0.10,0.09,0.49}{##1}}}
\@namedef{PY@tok@sx}{\def\PY@tc##1{\textcolor[rgb]{0.00,0.50,0.00}{##1}}}
\@namedef{PY@tok@m}{\def\PY@tc##1{\textcolor[rgb]{0.40,0.40,0.40}{##1}}}
\@namedef{PY@tok@gh}{\let\PY@bf=\textbf\def\PY@tc##1{\textcolor[rgb]{0.00,0.00,0.50}{##1}}}
\@namedef{PY@tok@gu}{\let\PY@bf=\textbf\def\PY@tc##1{\textcolor[rgb]{0.50,0.00,0.50}{##1}}}
\@namedef{PY@tok@gd}{\def\PY@tc##1{\textcolor[rgb]{0.63,0.00,0.00}{##1}}}
\@namedef{PY@tok@gi}{\def\PY@tc##1{\textcolor[rgb]{0.00,0.52,0.00}{##1}}}
\@namedef{PY@tok@gr}{\def\PY@tc##1{\textcolor[rgb]{0.89,0.00,0.00}{##1}}}
\@namedef{PY@tok@ge}{\let\PY@it=\textit}
\@namedef{PY@tok@gs}{\let\PY@bf=\textbf}
\@namedef{PY@tok@gp}{\let\PY@bf=\textbf\def\PY@tc##1{\textcolor[rgb]{0.00,0.00,0.50}{##1}}}
\@namedef{PY@tok@go}{\def\PY@tc##1{\textcolor[rgb]{0.44,0.44,0.44}{##1}}}
\@namedef{PY@tok@gt}{\def\PY@tc##1{\textcolor[rgb]{0.00,0.27,0.87}{##1}}}
\@namedef{PY@tok@err}{\def\PY@bc##1{{\setlength{\fboxsep}{\string -\fboxrule}\fcolorbox[rgb]{1.00,0.00,0.00}{1,1,1}{\strut ##1}}}}
\@namedef{PY@tok@kc}{\let\PY@bf=\textbf\def\PY@tc##1{\textcolor[rgb]{0.00,0.50,0.00}{##1}}}
\@namedef{PY@tok@kd}{\let\PY@bf=\textbf\def\PY@tc##1{\textcolor[rgb]{0.00,0.50,0.00}{##1}}}
\@namedef{PY@tok@kn}{\let\PY@bf=\textbf\def\PY@tc##1{\textcolor[rgb]{0.00,0.50,0.00}{##1}}}
\@namedef{PY@tok@kr}{\let\PY@bf=\textbf\def\PY@tc##1{\textcolor[rgb]{0.00,0.50,0.00}{##1}}}
\@namedef{PY@tok@bp}{\def\PY@tc##1{\textcolor[rgb]{0.00,0.50,0.00}{##1}}}
\@namedef{PY@tok@fm}{\def\PY@tc##1{\textcolor[rgb]{0.00,0.00,1.00}{##1}}}
\@namedef{PY@tok@vc}{\def\PY@tc##1{\textcolor[rgb]{0.10,0.09,0.49}{##1}}}
\@namedef{PY@tok@vg}{\def\PY@tc##1{\textcolor[rgb]{0.10,0.09,0.49}{##1}}}
\@namedef{PY@tok@vi}{\def\PY@tc##1{\textcolor[rgb]{0.10,0.09,0.49}{##1}}}
\@namedef{PY@tok@vm}{\def\PY@tc##1{\textcolor[rgb]{0.10,0.09,0.49}{##1}}}
\@namedef{PY@tok@sa}{\def\PY@tc##1{\textcolor[rgb]{0.73,0.13,0.13}{##1}}}
\@namedef{PY@tok@sb}{\def\PY@tc##1{\textcolor[rgb]{0.73,0.13,0.13}{##1}}}
\@namedef{PY@tok@sc}{\def\PY@tc##1{\textcolor[rgb]{0.73,0.13,0.13}{##1}}}
\@namedef{PY@tok@dl}{\def\PY@tc##1{\textcolor[rgb]{0.73,0.13,0.13}{##1}}}
\@namedef{PY@tok@s2}{\def\PY@tc##1{\textcolor[rgb]{0.73,0.13,0.13}{##1}}}
\@namedef{PY@tok@sh}{\def\PY@tc##1{\textcolor[rgb]{0.73,0.13,0.13}{##1}}}
\@namedef{PY@tok@s1}{\def\PY@tc##1{\textcolor[rgb]{0.73,0.13,0.13}{##1}}}
\@namedef{PY@tok@mb}{\def\PY@tc##1{\textcolor[rgb]{0.40,0.40,0.40}{##1}}}
\@namedef{PY@tok@mf}{\def\PY@tc##1{\textcolor[rgb]{0.40,0.40,0.40}{##1}}}
\@namedef{PY@tok@mh}{\def\PY@tc##1{\textcolor[rgb]{0.40,0.40,0.40}{##1}}}
\@namedef{PY@tok@mi}{\def\PY@tc##1{\textcolor[rgb]{0.40,0.40,0.40}{##1}}}
\@namedef{PY@tok@il}{\def\PY@tc##1{\textcolor[rgb]{0.40,0.40,0.40}{##1}}}
\@namedef{PY@tok@mo}{\def\PY@tc##1{\textcolor[rgb]{0.40,0.40,0.40}{##1}}}
\@namedef{PY@tok@ch}{\let\PY@it=\textit\def\PY@tc##1{\textcolor[rgb]{0.24,0.48,0.48}{##1}}}
\@namedef{PY@tok@cm}{\let\PY@it=\textit\def\PY@tc##1{\textcolor[rgb]{0.24,0.48,0.48}{##1}}}
\@namedef{PY@tok@cpf}{\let\PY@it=\textit\def\PY@tc##1{\textcolor[rgb]{0.24,0.48,0.48}{##1}}}
\@namedef{PY@tok@c1}{\let\PY@it=\textit\def\PY@tc##1{\textcolor[rgb]{0.24,0.48,0.48}{##1}}}
\@namedef{PY@tok@cs}{\let\PY@it=\textit\def\PY@tc##1{\textcolor[rgb]{0.24,0.48,0.48}{##1}}}

\def\PYZbs{\char`\\}
\def\PYZus{\char`\_}
\def\PYZob{\char`\{}
\def\PYZcb{\char`\}}
\def\PYZca{\char`\^}

\def\PYZsh{\char`\#}

\def\PYZhy{\char`\-}
\def\PYZsq{\char`\'}
\def\PYZdq{\char`\"}

\makeatletter
\newbox\Wrappedcontinuationbox 
\newbox\Wrappedvisiblespacebox 
\newcommand*\Wrappedvisiblespace {\textcolor{red}{\textvisiblespace}} 
\newcommand*\Wrappedcontinuationsymbol {\textcolor{red}{\llap{\tiny$\m@th\hookrightarrow$}}} 
\newcommand*\Wrappedcontinuationindent {3ex } 
\newcommand*\Wrappedafterbreak {\kern\Wrappedcontinuationindent\copy\Wrappedcontinuationbox} 
\newcommand*\Wrappedbreaksatspecials {%
	\def\PYGZus{\discretionary{\char`\_}{\Wrappedafterbreak}{\char`\_}}%
	\def\PYGZob{\discretionary{}{\Wrappedafterbreak\char`\{}{\char`\{}}%
	\def\PYGZcb{\discretionary{\char`\}}{\Wrappedafterbreak}{\char`\}}}%
	\def\PYGZca{\discretionary{\char`\^}{\Wrappedafterbreak}{\char`\^}}%
	\def\PYGZam{\discretionary{\char`\&}{\Wrappedafterbreak}{\char`\&}}%
	\def\PYGZlt{\discretionary{}{\Wrappedafterbreak\char`\<}{\char`\<}}%
	\def\PYGZgt{\discretionary{\char`\>}{\Wrappedafterbreak}{\char`\>}}%
	\def\PYGZsh{\discretionary{}{\Wrappedafterbreak\char`\#}{\char`\#}}%
	\def\PYGZpc{\discretionary{}{\Wrappedafterbreak\char`\%}{\char`\%}}%
	\def\PYGZdl{\discretionary{}{\Wrappedafterbreak\char`\$}{\char`\$}}%
	\def\PYGZhy{\discretionary{\char`\-}{\Wrappedafterbreak}{\char`\-}}%
	\def\PYGZsq{\discretionary{}{\Wrappedafterbreak\textquotesingle}{\textquotesingle}}%
	\def\PYGZdq{\discretionary{}{\Wrappedafterbreak\char`\"}{\char`\"}}%
	\def\PYGZti{\discretionary{\char`\~}{\Wrappedafterbreak}{\char`\~}}%
} 
\newcommand*\Wrappedbreaksatpunct {%
	\lccode`\~`\.\lowercase{\def~}{\discretionary{\hbox{\char`\.}}{\Wrappedafterbreak}{\hbox{\char`\.}}}%
	\lccode`\~`\,\lowercase{\def~}{\discretionary{\hbox{\char`\,}}{\Wrappedafterbreak}{\hbox{\char`\,}}}%
	\lccode`\~`\;\lowercase{\def~}{\discretionary{\hbox{\char`\;}}{\Wrappedafterbreak}{\hbox{\char`\;}}}%
	\lccode`\~`\:\lowercase{\def~}{\discretionary{\hbox{\char`\:}}{\Wrappedafterbreak}{\hbox{\char`\:}}}%
	\lccode`\~`\?\lowercase{\def~}{\discretionary{\hbox{\char`\?}}{\Wrappedafterbreak}{\hbox{\char`\?}}}%
	\lccode`\~`\!\lowercase{\def~}{\discretionary{\hbox{\char`\!}}{\Wrappedafterbreak}{\hbox{\char`\!}}}%
	\lccode`\~`\/\lowercase{\def~}{\discretionary{\hbox{\char`\/}}{\Wrappedafterbreak}{\hbox{\char`\/}}}%
	\catcode`\.\active
	\catcode`\,\active 
	\catcode`\;\active
	\catcode`\:\active
	\catcode`\?\active
	\catcode`\!\active
	\catcode`\/\active 
	\lccode`\~`\~ 	
}
\makeatother

\let\OriginalVerbatim=\Verbatim
\makeatletter
\renewcommand{\Verbatim}[1][1]{%
	\sbox\Wrappedcontinuationbox {\Wrappedcontinuationsymbol}%
	\sbox\Wrappedvisiblespacebox {\FV@SetupFont\Wrappedvisiblespace}%
	\def\FancyVerbFormatLine ##1{\hsize\linewidth
		\vtop{\raggedright\hyphenpenalty\z@\exhyphenpenalty\z@
			\doublehyphendemerits\z@\finalhyphendemerits\z@
			\strut ##1\strut}%
	}%
	\def\FV@Space {%
		\nobreak\hskip\z@ plus\fontdimen3\font minus\fontdimen4\font
		\discretionary{\copy\Wrappedvisiblespacebox}{\Wrappedafterbreak}
		{\kern\fontdimen2\font}%
	}%
	
	\Wrappedbreaksatspecials
	\OriginalVerbatim[#1,codes*=\Wrappedbreaksatpunct]%
}
\makeatother

\definecolor{incolor}{HTML}{303F9F}
\definecolor{outcolor}{HTML}{D84315}
\definecolor{cellborder}{HTML}{CFCFCF}
\definecolor{cellbackground}{HTML}{F7F7F7}

\makeatletter
\newcommand{\boxspacing}{\kern\kvtcb@left@rule\kern\kvtcb@boxsep}
\makeatother
\newcommand{\prompt}[4]{
	{\ttfamily\llap{{\color{#2}[#3]:\hspace{3pt}#4}}\vspace{-\baselineskip}}
}

\makeatletter
\newcommand*{\owedge}{%
	\mathbin{%
		\mathpalette\@owedge{}%
	}%
}
\newcommand*{\@owedge}[2]{%
	\sbox0{$#1\oplus\m@th$}%
	\dimen2=.5\dimexpr\wd0-\ht0-\dp0\relax 
	\dimen@=\dimexpr\ht0+\dp0\relax
	\def\lw{.04}
	\def\radius{.5-\lw/2}%
	\kern\dimen2 
	\tikz[
	line width=\lw\dimen@,
	line join=round,
	x=\dimen@,
	y=\dimen@,
	baseline=\dimexpr-.5\dimen@+\dp0\relax,
	]
	\draw
	(0,0) circle[radius=\radius]
	(225:\radius) -- (0,.5-\lw) -- (-45:\radius)
	;%
	\kern\dimen2 
}	
\newtheorem{theorem}{Theorem}
\newtheorem{definition}{Definition}
\newtheorem{proposition}{Proposition}
\newtheorem{lemma}{Lemma}

\title{Introduction to Lorentzian and Flat Affine Geometry of $\mathsf{GL}(2,\mathbb{R})$}
\author{Alberto Medina.\\
	{\small Université de Montpellier, Institute A. Grothendieck, France}\\
{\small Email: alberto.medina@umontpellier.fr}\\\vspace{0.5cm}
	{\small Email: albertomdn47@gmail.com}\\
 E. Andr\'es Villab\'on.\\
	{\small Universidad Nacional Abierta y a Distancia, Colombia}\\
	{\small Email: edgar.villabon@unad.edu.co}
}
\date{}
\begin{document}
\maketitle
\begin{abstract}
	
	The goal of this paper is to study the geometry of the connected unit component of the real general linear Lie group $4$ dimensional $G_0$ as a Lorentzian and flat affine manifold. \\
	As the group $G_0$ is naturally equipped with a bi-invariant Hessian metric $k^+$, relative to a bi-invariant flat affine structure $\nabla$, we examine these structures and the relationships between them. Both structures are defined using the Lie algebra $\mathfrak{g}$, the first one through the trace $k(u,v):=\mathrm{trace}(u\circ v)$ and the second by the composition $\nabla_{u^+}v^+:=(u\circ v)^+$, where $u,v\in\mathfrak{g}$.\\
	The curvatures, tidal force, and Jacobi vector fields of $(G_0, k^+)$ are determined in Section 1. Section 2 discusses the causal structure of $(G_0,k^+)$, while Section 3 focuses on the developed map relative to $\nabla$ in the sense of C. Ehresmann.
	
\end{abstract}
 Keywords: Lorentzian metric, flat affine structure, quadratic Lie groups, curvatures, Weyl tensor.
\section{Lorentzian Structure}
Let  $G_0$ be the connected unit component of the Lie group of isomorphisms of the real vector space  $\mathbb{R}^2$ and $\mathfrak{g}$ its Lie algebra. We will identify $\mathfrak{g}$ with $T_\varepsilon(G_0)$, the tangent space of $G_0$ in the identity element $\varepsilon$, or with the space of left-invariant vector fields on $G_0$.

Consider the left-invariant metric $k^+$ on $G_0$ given by the inner product $k$ on $\mathfrak{g}$ defined by 
$$k(u,v):=\mathrm{trace}(u\circ v),\text{ where }u,v\in\mathfrak{g}.$$
If $u^+$ is the left-invariant vector field on $G_0$ determined by $u\in\mathfrak{g}$, since the operators $ad_{u^+}$ are $k^+$-antisymmetric, the right translations on $G_0$ are also isometries. Then  $k^+$ is a bi-invariant metric, i.e. $(G_0, k^+)$ is a quadratic Lie group in the Medina-Revoy's sense.

Consider the following $k$-orthonormal basis $\mathcal{B}$ of $\mathfrak{g}$ given by
\begin{align}
	e_1&:=\frac{\sqrt2}{2}\left(\begin{array}{cc}
		0 &1\\
		-1&0
	\end{array}\right), &
	e_2&:=\frac{\sqrt2}{2}\left(\begin{array}{cc}
		0 &1\\
		1&0
	\end{array}\right), \label{baseortogonal}\\
	\notag
	e_3&:=\frac{\sqrt2}{2}\left(\begin{array}{cc}
		1 &0\\
		0&-1
	\end{array}\right),&
	e_4&:=\frac{\sqrt2}{2}\left(\begin{array}{cc}
		1&0\\
		0&1
	\end{array}\right).
\end{align}
The Lie brackets, except antisymmetries, relatives to $\mathcal{B}$ become
$$
[e_1,e_2]=\sqrt{2}e_3,\;[e_1,e_3]=-\sqrt{2}e_2,\;[e_2,e_3]=-\sqrt{2}e_1.
$$
As $k$  is given by the table
$$ \begin{array}{r|c|c|c|c}
	k&e_1&e_2&e_3&e_4\\
	\hline
	e_1 &-1&0&0&0\\
	\hline
	e_2& 0&1&0&0\\
	\hline
	e_3& 0&0&1&0\\
	\hline
	e_4&0&0&0&1
\end{array} $$
 $k^+$ is Lorentzian.

Recall that, a tangent vector $u$ to $G_0$ is called timelike (respectively lightlike, spacelike) if $k(u,u)<0$ (respectively $k(u,u)=0$, $k(u,u)>0$).  In particular, the vector $e_1$ is timelike, and the others $e_i$ are spacelike.

We will show that every left-invariant semi-Riemannian metric on $G_0$ is non-flat.

The following assertion is clear. 
\begin{lemma}\label{BrombergMedina} Any left invariant semi-Riemannian metric on a quadratic Lie group $(H, g)$ is given by a linear isomorphism $g$-symmetric $\varphi$ on its Lie algebra $\mathfrak{h}$ by the formula
	$$\langle u,v\rangle_\varphi:=g(\varphi(u),v)),\;u,v\in\mathfrak{h}$$
\end{lemma} 

From Lemma \ref{BrombergMedina}, we have 
\begin{proposition}
	Every left-invariant (respectively right-invariant) semi-Riemannian metric on $G_0$ is not flat.
\end{proposition}
The proposition is a consequence of the following lemma
\begin{lemma}  
	The left-invariant semi-Riemannian metrics on $G_0$ are given by the following $k$-symmetric linear isomorphisms of $\mathfrak{g}$ 
	\begin{align*}
		\varphi_0(e_1)&=-e_1,\; \varphi_0(e_2)=e_2,\; \varphi_0(e_3)=e_3,\; \varphi_0(e_4)=e_4\\
		\varphi_1(e_1)&=e_1,\; \varphi_1(e_2)=e_2,\; \varphi_1(e_3)=e_3,\; \varphi_1(e_4)=e_4\\
		\varphi_2(e_1)&=e_1,\; \varphi_2(e_2)=-e_2,\; \varphi_2(e_3)=e_3,\; \varphi_2(e_4)=e_4
	\end{align*}
	and their index are $0, 1, 2$ respectively. Moreover, the curvatures of  $<u,v>_i:=k(\varphi_i(u),v)$ for $i=0,1,2$ are not flat. 
\end{lemma}

These computations are provided in Section \ref{python}. 

For the study of the movement space in $(G_0, k^+)$, in addition to the tangent bundle, it is convenient to introduce some principal fiber bundles. For instance, the orthonormal fiber bundle over $(G_0, k^+)$ facilitates the understanding of parallel transport along a curve, in particular the holonomy groups and the corresponding holonomy bundles.

Let $P= L(G_0)$ be the principal fiber bundle of the linear frames over $G_0.$ Since $\mathsf{GL}(4,\mathbb{R})$ is $16$-dimensional, $P$ has a dimension of $20.$  On the other hand, as the orthogonal group $O(1,3)$ is the dimension $6$, the principal fiber bundle of the orthogonal frames over $G_0$ is $10$-dimensional.

Now, the Levi-Civita connection $D$ of $(G_0, k^+)$ can be expressed using left-invariant fields as $$D_{u^+}v^+:=\frac{1}{2}[u,v]^+$$ 
so the curvature, sectional curvature, scalar curvature, and Ricci tensors become
\begin{align*}
	R_{u^+v^+}w^+&:=\frac{1}{4}\left[\left[u^+,v^+\right],w^+\right]\\
	K(u^+,v^+)&:=\frac{1}{4}\frac{k^+\left(\left[u^+,v^+\right],\left[u^+,v^+\right]\right)}{k^+\left(u^+,u^+\right)k^+\left(v^+,v^+\right)- k^+\left(u^+,v^+\right)^2}
\end{align*}
where $\mathsf{span}\{u^+,v^+\}$ is a nondegenerate plane,
\begin{align*}
	S&:=2\sum_{i<j}K(e_i^+,e_j^+)\\
	R_{icci}(u,v)&:=-\frac{B(u,v)}{4}=-\frac{1}{4}\mathrm{trace}\left(\mathsf{ad}_u\circ\mathsf{ad}_v\right)
\end{align*}
respectively. Here $B$ is known as the Cartan-Killing form. 

Consequently, we have 
\begin{lemma}
	In terms of the above orthonormal basis \eqref{baseortogonal} of $(\mathfrak{g},k)$, the curvature tensor, the sectional curvature, the scalar curvature, the Cartan-Killing form and the Ricci tensor of $(G_0, k^+, D)$  are given respectively by 
	\begin{gather*}
		R\left(e_1^+,e_2^+,e_1^+\right)=\frac{1}{2}e_2^+,R\left(e_1^+,e_2^+,e_2^+\right)=\frac{1}{2}e_1^+ , R\left(e_1^+,e_3^+,e_1^+\right)=\frac{1}{2}e_3^+, \\
		R\left(e_1^+,e_3^+,e_3^+\right)=\frac{1}{2}e_1^+,
		R\left(e_2^+,e_1^+,e_1^+\right)=-\frac{1}{2}e_2^+, R\left(e_2^+,e_1^+,e_2^+\right)=-\frac{1}{2}e_1^+,\\
		R\left(e_2^+,e_3^+,e_2^+\right)=-\frac{1}{2}e_3^+,  R\left(e_2^+,e_3^+,e_3^+\right)=\frac{1}{2}e_2^+,  R\left(e_3^+,e_1^+,e_1^+\right)=-\frac{1}{2}e_3^+,\\
		R\left(e_3^+,e_1^+,e_3^+\right)=-\frac{1}{2}e_1^+ , R\left(e_3^+,e_2^+,e_2^+\right)=e_3^+,R\left(e_3^+,e_2^+,e_3^+\right)=-\frac{1}{2}e_2^+\\
		R\left(e_i^+,e_j^+,e_k^+\right)=0 \text{ in the other cases.} 
	\end{gather*}
	\begin{gather*}
		K(e_1^+,e_2^+)= K(e_1^+,e_3^+)=K(e_2^+,e_3^+)=-\frac{1}{2},\\
		K(e_i^+,e_4^+)=0,\;\text{ for }i=1,2,3,\\
		S=-3\\
		B(e_1,e_1)=-4,\;B(e_2,e_2)=B(e_3,e_3)=4,\;B(e_4,e_4)=0\\
		B(e_i,e_j)=0\;\text{ for }i\neq j\\
		R_{icci}(e_1^+,e_1^+)=1,\; R_{icci}(e_2^+,e_2^+)=R_{icci}(e_3^+,e_3^+)=-1,\\
		R_{icci}(e_4^+,e_4^+)=R_{icci}(e_i^+,e_j^+)=0,\;\text{ for }i\neq j
	\end{gather*}	
\end{lemma} 
It is clear that the curvature tensor (respectively the tensor Ricci), viewed in the principal fiber bundle of the orthogonal frames, has $12$ components (respectively $3$).

Recall that the structure of the gravitational field,  without matter (or antimatter), is determined by the Weyl tensor. The next result exhibited the components of the Weyl tensor.

\begin{theorem} Using the orthonormal basis  \eqref{baseortogonal}, the Weyl tensor components are given by:
	{\small\begin{align*}
			W(e_i^+,e_j^+,e_k^+,e_l^+)=\left\{\begin{array}{cl}
				0&\text{if } i=j=k=l,\;  i=j=k\neq l,\; i\neq j\neq k\neq l\\
				&\\
				\frac{3}{2}&\text{if } i=k=1, j=l=2,\text{ or } i=k=1, j=l=3,\\
				&\text{ or } i=l=2, j=k=3\\
				&\\
				-\frac{3}{2}&\text{if } i=l=1, j=k=2,\text{ or } i=l=1, j=k=3,\\
				& \text{ or } i=k=2, j=l=3\\
				&\\
				\frac{1}{2}&\text{if } i=l=2,3, j=k=4,\text{ or } i=l=4, j=k=2,3,\\
				&\text{ or } i=k=1, j=l=4, \text{ or } i=k=4, j=l=1\\
				&\\
				-\frac{1}{2}&\text{if } i=l=1, j=k=4,\text{ or } i=l=4, j=k=1,\\
				&\text{ or } i=k=2,3, j=l=4, \text{ or } i=k=4, j=l=2,3
			\end{array}\right.
	\end{align*}}
\end{theorem}
\begin{proof}
	The Kulkarni–Nomizu product of two symmetric covariant 2-tensor $h$ and $l$, can be obtained by:
	\begin{align}\label{Kulkarni}
		(h\owedge l)(w,x,y,z):=&h(w,z)l(x,y)+h(x,y)l(w,z)\\
		&-h(w,y)l(x,z)-h(x,z)l(w,y)\notag
	\end{align}
	and
	\begin{align*}
		W&=R_m-\frac{1}{2}R_{icci}\owedge k^++\frac{S}{12}k^+\owedge k^+\\
		&=R_m-\frac{1}{2}R_{icci}\owedge k^+-\frac{1}{4}k^+\owedge k^+
	\end{align*}
	where $R_m(X_1,X_2,X_3,X_4)=\langle R(X_1,X_2,X_3),X_4\rangle_{k^+}$ for all $X_1,X_2,X_3,X_4\in\mathfrak{X}(G_0)$ (see \cite{Le}), the result is followed.
\end{proof}

Working in the (natural) coordinates system 
\begin{equation}\label{sistemcoor1}
	(x_1,x_2,x_3,x_4)\equiv
	\left(\begin{array}{cc}
		x_1&x_2\\
		x_3&x_4
	\end{array}
	\right)\end{equation}
we have a smooth frame   
\begin{align*}
	e_1^+&=\frac{\sqrt{2}}{2}\left(-x_2\frac{\partial}{\partial x_1}+x_1\frac{\partial}{\partial x_2}-x_4\frac{\partial}{\partial x_3}+x_3\frac{\partial}{\partial x_4}\right)\\
	e_2^+&=\frac{\sqrt{2}}{2}\left(x_2\frac{\partial}{\partial x_1}+x_1\frac{\partial}{\partial x_2}+x_4\frac{\partial}{\partial x_3}+x_3\frac{\partial}{\partial x_4}\right)\\
	e_3^+&=\frac{\sqrt{2}}{2}\left(x_1\frac{\partial}{\partial x_1}-x_2\frac{\partial}{\partial x_2}+x_3\frac{\partial}{\partial x_3}-x_4\frac{\partial}{\partial x_4}\right)\\
	e_4^+&=\frac{\sqrt{2}}{2}\left(x_1\frac{\partial}{\partial x_1}+x_2\frac{\partial}{\partial x_2}+x_3\frac{\partial}{\partial x_3}+x_4\frac{\partial}{\partial x_4}\right)
\end{align*} 
and its corresponding dual coframe
\begin{align*}
	\left(e_1^+\right)^*&=\frac{1}{\sqrt{2}(x_1x_4 - x_2x_3)}\left(x_3dx_1+x_4dx_2-x_1dx_3-x_2dx_4\right)\\
	\left(e_2^+\right)^*&=\frac{1}{\sqrt{2}(x_1x_4 - x_2x_3)}(-x_3dx_1+x_4dx_2+
	x_1dx_3-x_2dx_4)\\
	\left(e_3^+\right)^*&=\frac{1}{\sqrt{2}(x_1x_4 - x_2x_3)}(x_4dx_1+x_3dx_2-
	x_2dx_3-x_1dx_4)\\
	\left(e_4^+\right)^*&=\frac{1}{\sqrt{2}(x_1x_4 - x_2x_3)}(x_4dx_1-x_3dx_2-
	x_2dx_3+x_1dx_4)
\end{align*}
 Then we have
\begin{proposition}
	In the coordinate system \eqref{sistemcoor1} the elements involved in Einstein's equation are written as  
	\begin{align*}
		k^+&=-\left(e_1^+\right)^*\left(e_1^+\right)^*+\left(e_2^+\right)^*\left(e_3^+\right)^*+\left(e_3^+\right)^*\left(e_3^+\right)^*+\left(e_4^+\right)^*\left(e_4^+\right)^*\\
		&=\frac{1}{2(x_1x_4-x_2x_3)^2}\left(2 x_{4}^{2} dx_{1}^{2} - 4  x_{3} x_{4}dx_{1} dx_{2} - 4 x_{2} x_{4}dx_{1} dx_{3}  \right.\\
		&+ 4 x_{2} x_{3} dx_{1} dx_{4} + 2 x_{3}^{2}dx_{2}^{2}+ 4  x_{1} x_{4}dx_{2} dx_{3} - 4 x_{1} x_{3}dx_{2} dx_{4}\\
		&\left.  + 2x_{2}^{2} dx_{3}^{2}  - 4 x_{1} x_{2}dx_{3} dx_{4}  + 2 x_{1}^{2}dx_{4}^{2}\right)
	\end{align*}
	\begin{align*}
		R_{icci}&=\left(- \frac{x_{1} x_{2}}{2} + \frac{x_{1} x_{4}}{2} + \frac{x_{2}^{2}}{2}\right)dx_1^2+ \left(\frac{\sqrt{2} x_{1}}{2} + \frac{\sqrt{2} x_{2}}{2}\right)dx_1dx_2\\
		&+\left(\frac{\sqrt{2} x_{1}}{2} + \frac{\sqrt{2} x_{4}}{2}\right)dx_1dx_3
		+\left(\frac{\sqrt{2} x_{1}}{2} + \frac{\sqrt{2} x_{3}}{2}\right)dx_1dx_4\\
		&+\left(\sqrt{2} x_{1}\right)dx_2^2+\left(- \frac{\sqrt{2} x_{2}}{2} + \frac{\sqrt{2} x_{4}}{2}\right)dx_2dx_3\\
		&+\left(\frac{\sqrt{2} x_{2}}{2} + \frac{\sqrt{2} x_{3}}{2}\right)dx_2dx_4+\left(\sqrt{2} x_{3}\right)dx_3^2\\
		&+\left(\frac{\sqrt{2} x_{3}}{2} - \frac{\sqrt{2} x_{4}}{2}\right)dx_3dx_4+\left(\sqrt{2} x_{4}\right)dx_4^2
	\end{align*}
\end{proposition}
\subsection{Tidal Force and Jacobi Vector Fields}
\begin{definition}
	For a vector $0\neq v\in T_\sigma(G_0)$ the tidal force operator $F_v: u^\perp\rightarrow v^\perp$ is given by $F_v(y) = R_{yv}v$. 
\end{definition}
The tidal force is a self-adjoint linear operator on $u^\perp$, and $\mathrm{trace}(F_v) =- R_{icci}(v, v)$.

A direct calculation gives 
\begin{lemma} 
	For $\displaystyle v_\sigma=\sum_{i=1}^4 f_ie_{i,\sigma}^+$, with $\sigma\in G_0$ we have 
	\begin{align*}
		\mathrm{trace}(F_{v_\sigma})&=- R_{icci}\left({v_\sigma},{v_\sigma}\right)\\
		&=- R_{icci}\left(\sum_{i=1}^4 f_ie_{i,\sigma}^+,\sum_{j=1}^4 f_je_{j,\sigma}^+\right)\\
		&=-\sum_{i,j=1}^4 f_if_j R_{icci}\left(e_{i,\sigma}^+,e_{j,\sigma}^+\right)\\
		&=-f_1^2+f_2^2+f_3^2
	\end{align*}
\end{lemma}  

\begin{definition} \textbf{Jacobi Fields.}
	
	If $\tau (t)$ is a geodesic, its variations in $(G_0,k^+)$ are given by vector fields $Y$ solutions of the differential equation
	
	\begin{equation}\label{Jacobi}
		\frac{D^2Y}{dt^2}= {\mathrm{R}}(Y,\tau',\tau')
	\end{equation}
	
	where $DY/dt$ denotes the affine covariant derivative relative to $\nabla$ along $\tau$. 
\end{definition}

Now, we will focus on the Jacobi equation in a Lie Group following \cite{BM2}.

Every vector field $X$ on $G_0$ defines a map
\begin{align*}
	\begin{array}{rrcl}
		\tilde{X}:&G_0&\to&{\mathfrak{g}}\\
		&\sigma&\mapsto&
		({\mathrm{L}}_{\sigma^{-1}})_{*,\sigma}X_\sigma.
	\end{array}
\end{align*}
Obviously, a vector field is left invariant if and only if the associated map is
constant.

Given a curve $\sigma :[t_0,t_1]\to G_0,$ every vector field $Y$ on
$\sigma$ defines a curve in ${\mathfrak{g}}:$
$$\tilde{Y}(t)=\left({\mathrm{L}}_{\sigma(t)^{-1}}\right)_{*,\sigma(t)}Y(t)
$$and conversely, every curve in ${\mathfrak{g}}$ defined on $[0,1]$ determines a
vector field on $\sigma.$ We say that one is the {\bf{reflection}}
of the other and we write either $y^\sim=Y$ or $Y^\sim=y.$

Notice that $y(t)=(Y(t))^\sim$ is equivalent to
$y(t)_{\sigma(t)}^+=Y(t).$

\begin{theorem}
	The reflection $y(t)=y_1e_1^++y_2e_2^++y_3e_3^++y_4e_4^+$ in $\mathfrak{g}$ of the Jacobi field along a geodesic $\sigma: [0,1]\to G$ with $\sigma(0) = \varepsilon$ and initial velocity $\dot{\sigma}(0)=ae_1+be_2+ce_3+de_4$ has components
	\begin{align*}
		y_1(t)&=A_0t+A_1e^{\alpha t}+A_2e^{-\alpha t}\\
		y_2(t)&=B_0+B_1e^{\alpha t}+A_2e^{-\alpha t}\\
		y_3(t)&=C_0+C_1e^{\alpha t}+C_2e^{-\alpha t}\\
		y_4(t)&=D_0+D_1t
	\end{align*}
	where $A_i, B_i, C_i, D_i$ and $\alpha$ are constants that depend on $a, b, c, d$.
\end{theorem}

\begin{proof}
	In fact, from the Jacobi equation, we obtain
	\begin{align*}
		{y''}_1&=-\sqrt{2}\;cy'_2+\sqrt{2}\;by'_3\\
		{y''}_2&=-\sqrt{2}\;cy'_1+\sqrt{2}\;ay'_3\\
		{y''}_3&=\sqrt{2}\;by'_1-\sqrt{2}\;ay'_2\\
		{y''}_4&=0
	\end{align*}
	And a direct calculation gives
	{\tiny
		\begin{align*}
			y_1(t)&=\left.\frac{1}{2 (a^2-b^2-c^2)}\right(2 a t (a C_1-b C_2-c C_3)+\\
			+&\frac{e^{-\sqrt{2} \sqrt{-a^2+b^2+c^2} t}}{\sqrt{2}\sqrt{-a^2+b^2+c^2}}  \left(b^2 C_1+c^2 C_1-a b C_2+c \sqrt{-a^2+b^2+c^2} C_2-a c C_3-b \sqrt{-a^2+b^2+c^2} C_3\right)+\\
			+&\left.\frac{ e^{\sqrt{2} \sqrt{-a^2+b^2+c^2} t}}{\sqrt{2} \sqrt{-a^2+b^2+c^2}}\left(-b^2 C_1-c^2 C_1+a b C_2+c \sqrt{-a^2+b^2+c^2} C_2+a c C_3-b \sqrt{-a^2+b^2+c^2}C_3\right)\right)\\
			y_2(t)&=\left.\frac{1}{2 (a^2-b^2-c^2)} \right(-2 b t (-a C_1+b C_2+c C_3)+\\
			+&\frac{e^{-\sqrt{2} \sqrt{-a^2+b^2+c^2 t} }}{\sqrt{2}\sqrt{-a^2+b^2+c^2}}\left (a b C_1+c \sqrt{-a^2+b^2+c^2} C_1-a^2 C_2+c^2 C_2-b c C_3-a \sqrt{-a^2+b^2+c^2} C_3\right)+\\
			+&\left.\frac{e^{\sqrt{2} \sqrt{-a^2+b^2+c^2} t} }{\sqrt{2} \sqrt{-a^2+b^2+c^2}} \left(-a b C_1+c \sqrt{-a^2+b^2+c^2} C_1+a^2 C_2-c^2 C_2+b c C_3-a \sqrt{-a^2+b^2+c^2} C_3\right)\right)\\
			y_3(t)&=\left.\frac{1}{2 (a^2-b^2-c^2)} \right(-2 c t (-a C_1+b C_2+c C_3)+\\
			+&\frac{e^{\sqrt{2} \sqrt{-a^2+b^2+c^2} t}}{\sqrt{2}\sqrt{-a^2+b^2+c^2}} \left(-a c C_1-b \sqrt{-a^2+b^2+c^2} C_1+b c C_2+a \sqrt{-a^2+b^2+c^2} C_2+a^2 C_3-b^2 C_3\right)+\\
			+&\left.\frac{e^{-\sqrt{2} \sqrt{-a^2+b^2+c^2} t}}{\sqrt{2} \sqrt{-a^2+b^2+c^2}} \left(a c C_1-b \sqrt{-a^2+b^2+c^2} C_1-b c C_2+a \sqrt{-a^2+b^2+c^2} C_2-a^2 C_3+b^2 C_3\right)\right)\\
			y_4(t)&=C_4t+C_5
		\end{align*}
	}
	
	In the case of $a^2=b^2+c^2$ the components of Jacobi fields become
	\begin{align*}
		y_1(t)=&C_1+\frac{1}{3}(3 t+b^2 t^3+c^2 t^3) C_2+\frac{1}{6} (-3 \sqrt{2} c t^2\pm2 b \sqrt{b^2+c^2} t^3) C_4\\
		&+\frac{1}{6} (3 \sqrt{2} b t^2\pm2 c \sqrt{b^2+c^2} t^3) C_6\\
		y_2(t)=&\frac{1}{6} (-3 \sqrt{2} c t^2\mp2 b \sqrt{b^2+c^2} t^3) C_2+C_3+\frac{1}{3} (3 t-b^2 t^3) C_4\\
		&+\frac{1}{6} (\mp3 \sqrt{2} \sqrt{b^2+c^2} t^2-2 b c t^3) C_6\\
		y_3(t)=&\frac{1}{6} (3 \sqrt{2} b t^2\mp2 c \sqrt{b^2+c^2} t^3)C_2+\frac{1}{6} (\mp3 \sqrt{2} \sqrt{b^2+c^2} t^2-2 b c t^3) C_4\\
		&+C_5+\frac{1}{3} (3 t-c^2 t^3) C_6\\
		y_4(t)=&C_7t+C_8
	\end{align*}
\end{proof}

\subsection{Parallel transport in terms of reflections}
The parallel transport is a tool that informs us about the holonomy of the Lorentzian manifold  $(G_0,k^+)$. 

A smooth vector field $X$ along a smooth curve $\sigma:[a,b]\to G_0$ is said to be parallel if $\frac{D X}{dt}= 0$. The map $P_a^b:T_{\sigma(a)}(G_0)\to T_{\sigma(b)}(G_0)$  sending each $v_{\sigma(a)}$ to a parallel vector $X$  on $\sigma$ such that $X(a)= v_{\sigma(a)}$ is called \textbf{parallel transport along $\sigma$}.

In what follows we give an outline of the parallel transport starting at $\varepsilon$. 

The integral curve of $e_1^+$ through $\varepsilon$ is
\begin{align*}
	\alpha(s)=\left(
	\begin{array}{cc}
		\cos\left(\frac{\sqrt{2}}{2}s\right) & \sin\left(\frac{\sqrt{2}}{2}s\right)\\
		-\sin\left(\frac{\sqrt{2}}{2}s\right) &\cos\left(\frac{\sqrt{2}}{2}s\right)
	\end{array}
	\right)
\end{align*}
Now, the lightlike vectors 
$u_\epsilon=\left(\begin{array}{cc} 
	a&b \\
	c&d
\end{array}\right)$  satisfies the condition $a^2+2bc+d^2=0$, and the integral curves through $\varepsilon$ are
\begin{align*}
	\alpha(s)=&e^{\frac{a+d}{2}s}
	\left(\begin{array}{cc}
		\cos\left(\frac{\theta}{2}s\right) +\frac{(a-d)}{\theta}\sin\left(\frac{\theta}{2}s\right)& \frac{2b}{\theta}\sin\left(\frac{\theta}{2}s\right)\\
		\frac{2c}{\theta}\sin\left(\frac{\theta}{2}s\right)&\cos\left(\frac{\theta}{2}s\right)-\frac{(a-d)}{\theta}\sin\left(\frac{\theta}{2}s\right)
	\end{array}
	\right)
\end{align*}
where $\theta=\operatorname{Re}(\sqrt{2bc-2ad})$.
Moreover, 
\begin{align*}
	\det(\alpha(s))&=e^{(a+d)s}\left(\cos\left(\frac{\theta}{2}s\right)+2\left(\frac{ad-bc}{\theta^2}\right)\sin\left(\frac{\theta}{2}s\right)\right)\\
	\mathrm{trace}(\alpha(s))&=e^{\frac{a+d}{2}s}\cos\left(\frac{\theta}{2}s\right)
\end{align*}
and the polynomial characteristic is 
$$p(\lambda)=\lambda^2+\operatorname{trace}(\alpha(s))\lambda+\det(\alpha(s))$$

Let $\gamma$ be a geodesic with initial direction $x_0$ starting in the unit of $G_0$, if $x=\displaystyle\sum_{i=1}^4 x_ie_i$ is the reflection of $\gamma'$,  the reflection $y$ of the parallel transport of $Y$ along $\gamma$ is given by the solution of the following system of differential equation
\begin{align*}
	y_1'&=\frac{\sqrt{2}}{2}\left(y_3x_2- y_2x_3 \right)\\
	y_2'&=\frac{\sqrt{2}}{2}\left(y_3x_1- y_1x_3 \right)\\
	y_3'&=\frac{\sqrt{2}}{2}\left(y_1x_2- y_2x_1 \right)\\
	y_4'&=0
\end{align*}

\subsection{Isometries of $(G_0,k^+)$}
Recall that we have
\begin{lemma}
	The left and right multiplication are isometries of $(G_0,k^+)$, and consequently the inversion map and the map $I_\sigma$,  defined by $I_\sigma(\tau):=\sigma\tau^{-1}\sigma$, for any $\sigma,\tau\in G_0$, are isometries. 
	
	Note that $I_\sigma$ inverse the geodesic, this means 
	$$I_\sigma(\gamma(s)):=\sigma\gamma(-s)\sigma$$
	where $\gamma(s)=e^{su}$ is a geodesic with $\gamma(0)=I_\varepsilon\in G_0$ and $\gamma'(0)=u\in\mathfrak{g}$
\end{lemma}

In particular for $\sigma=\varepsilon$ and  $u=e_1$
$$I_\varepsilon(\gamma(s))=
\left(
\begin{array}{cc}
	\cos\left(\frac{\sqrt{2}}{2}s\right) & -\sin\left(\frac{\sqrt{2}}{2}s\right) \\
	\sin\left(\frac{\sqrt{2}}{2}s\right) & \cos\left(\frac{\sqrt{2}}{2}s\right) 
\end{array}
\right)
$$

On the other hand, it is well-known that the metric $k^+$ can be lifted to the Lorentzian metric $\widetilde{k}^+$ over the universal covering $\widetilde{G}$ such that the covering map $p$ is a local isometry.  For a better comprehension of the isometries of $(G_0, k^+)$ it is convenient to study the isometries of the universal covering space $(\tilde{G}, \tilde{k}^+)$ (for instance see Proposition 2.1 and Theorem 2.2 of  \cite{BrMe}).

\subsection{Causal Structure}
Since $(G_0, k^+)$ is quadratic, its causal structure is determined by the Minkowski space $(T_\varepsilon (G_0),k)$.\\
Determining the different types of vectors constitutes the first step to comprehend a spacetime model. We will find these types of vectors. \\
Note that the matrix 
$\left(
\begin{array}{cc}
	a &b\\
	&\\
	c& d
\end{array}
\right)_\varepsilon$ 

given a vector of lightlike (respectively timelike, spacelike) if it verifies $a^2+2bc+d^2=0$ (respectively  $a^2+2bc+d^2<0$, $a^2+2bc+d^2>0$). \textit{It is commonly accepted that the set of all lightlike vectors in $T_\sigma(G_0)$ is the lightcone at $\sigma\in G_0$.}

Let $\mathcal{F}$ be the set of all timelike vectors in $\mathfrak{g}$. The timecone of $\mathfrak{g}$ containing $u\in\mathfrak{g}$ is defined as $C(u):=\{v\in\mathcal{F}|\; k(u,v)<0\}$, and the opposite timecone is  $-C(u):=\{v\in\mathcal{F}|\; k(u,v)>0\}$, so the timecone containing $e_1$  is 
$$C(e_1)=\left\{
\left.\left(
\begin{array}{cc}
	a &b\\
	&\\
	c& d
\end{array}\right)\right| c<b,\; a^2+2bc+d^2<0 
\right\}$$
and the opposite timecone is
$$-C(e_1)=\left\{
\left.\left(
\begin{array}{cc}
	a &b\\
	&\\
	c& d
\end{array}\right)\right| c>b,\; a^2+2bc+d^2<0 
\right\}$$
Note that the function $f$ from matrices $2\times2$ to real numbers defined as $f\left(a_{ij}\right)=a_{11}^2+2a_{12}a_{21}+a_{22}^2$ is a smooth map with constant rank 1, then $f^{-1}(0)$ is a embedded submanifold of dimension 4. \\
The following results are well known,
\begin{lemma}
	Timelike vectors $u$ and $v$ in a Lorentz vector space are in the same timecone if and only if $k(u, v) < 0$. Furthermore, timecones are convex.
\end{lemma}

\section{Flat Affine Structure on $\mathsf{GL}(2,\mathbb{R})_0$}
\begin{proposition}
	Since the matrix multiplication is associative, the above Lie group $G_0$  is endowed with a bi-invariant flat affine structure $\nabla$ i.e. a linear connection without torsion and curvature. 
\end{proposition}  
The universal covering group of $G_0$ is the product manifold  $\widetilde{G}=\mathbb{R}\times\mathsf{SDP}(2)$, where $\mathsf{SDP}(2)$ are the symmetric definite positive matrices and the covering map is
\begin{align}\label{covermap}
	\begin{array}{rccl}
		p:&\mathbb{R}\times\mathsf{SDP}(2)&\rightarrow&G_0\\
		&(t,T)&\mapsto&O_tT
	\end{array}
\end{align}
here $O_t$ denotes an orthogonal matrix. 	

The product of $G_0$ rises to a product in $\widetilde{G}$ such that $\widetilde{G}$ is a Lie group and $p$ is a homomorphism of Lie groups (see \cite{Ch} pag. 53)

To find the multiplication in $\widetilde{G}$, we must calculate the polar decomposition of $O_tTO_rR$, this is
$$\underbrace{O_tTO_rR(RO_{-r}T^2O_rR)^{-\frac{1}{2}}}_{O_s}\underbrace{(RO_{-r}T^2O_rR)^{\frac{1}{2}}}_S$$ 
Then the multiplication of two elements in $\widetilde{G}$ is given as follows:
\begin{align}\label{producto}
	(t,T)\cdot(r,R)=\left(s(t,r),(RO_{-r}T^2O_rR)^{\frac{1}{2}}\right)
\end{align}
where $s(t,r)\in\mathbb{R}$ can be calculated using the orthogonal matrix of order $2\times2$
$$
O_{s(t,r)}=O_tTO_rR(RO_{-r}T^2O_rR)^{-\frac{1}{2}}
$$
The multiplication \eqref{producto} can be written as
\begin{align}\label{producto2}
	(t,T)\cdot(r,R)=(\arctan(\theta)+t+r,\;O_{-\arctan(\theta)-t-r}TO_rR)
\end{align}
where
\begin{gather*}
	-\frac{\pi}{2}<\theta=\frac{\mathrm{trace}(O_{-r}TO_rR\omega)}{\mathrm{trace}(O_{-r}TO_rR)}<\frac{\pi}{2}
	\quad\mbox{and}\quad
	\omega=\left(\begin{array}{cc}
		0&1\\
		-1&0
	\end{array}\right)
\end{gather*}
A direct calculation shows that $\left(\widetilde{G}, \cdot\right)$ is a Lie group.

\begin{theorem} 
	The developed map of the flat affine manifold $(G_0,\nabla)$ in the Ehresmann's sense \cite{Eh}, following Koszul method \cite{K},  is given by
	$$
	\begin{array}{rcll}
		Dev:&\widetilde{G}_0&\rightarrow&\mathbb{R}^4\\
		&(y_1,y_2,y_3,y_4)&\mapsto&(y_1,y_2-1,y_3,y_4-1)
	\end{array}
	$$
\end{theorem} 
\begin{proof} Let  $(x_1,x_2,x_3,x_4)$ be the natural coordinates \eqref{sistemcoor1} of $G_0$. We know that $\nabla$ is locally isomorphic to usual connection $\nabla^0$ on $\mathbb{R}^4,$ then its 1-form connection $\omega$ can be written in local coordinates $\left(x_1,x_2,x_3,x_4,(X)_{ij})\right)$, like
	$$\omega=\left(\omega_{ij}\right)=\left( \sum_{k=1}^4Y_{ik}\operatorname{d} X_{kj}\right),\quad\mbox{where $(Y_{ij})=(X_{ij})^{-1}$}$$ 
	
	Let $(\widetilde{G},p)$ be the universal covering group of $G_0$, where $p$  is the covering map \eqref{covermap}, and $\widetilde{\mathfrak{g}}$ the Lie algebra of $\widetilde{G}$.   
	
	We can lift the affine structure to $\widetilde{G}$ through $\widetilde{\omega}:=p^*\omega$. 
	Now, we will compute the developed map of $(\widetilde{G},\widetilde{\nabla})$. First, we find the vector fields $Y$ on $\widetilde{G}$ such that $\widetilde{\nabla}Y=0$. In the local coordinates
	$$\left(y_1,\left(\begin{array}{cc}
		y_2&y_3\\
		y_3&y_4
	\end{array}\right)\right)\cong(y_1,y_2,y_3,y_4)$$ they  are
	$$Y=\sum_{i=1}^4a_i\frac{\partial}{\partial y_i},\quad\mbox{with $a_i\in\mathbb{R}$}$$
	Second,  we shall calculate the close $\widetilde{\mathfrak{g}}$-valued 1-form $\eta$ on $\widetilde{G}$ such that $\widetilde{\nabla}\eta=0$ and $\eta(Y)=Y_{\widetilde{\varepsilon}}$. A direct calculation shows that $\eta=(dy_1, dy_2,dy_3,dy_4)$
	
	Finally, the developed map  is given by $Dev(\widetilde{\sigma})=\int_C\eta$, where $C$ is a curve joining $\widetilde{\varepsilon}=(0,1,0,1)$ and $\widetilde{\sigma}$. Then
	$$
	\begin{array}{rccl}
		Dev:&\widetilde{G}_0&\rightarrow&\mathbb{R}^4\\
		&(y_1,y_2,y_3,y_4)&\mapsto&(y_1,y_2-1,y_3,y_4-1)
	\end{array}
	$$
	
\end{proof}
We have the following result (See \cite{AuMe}) 
\begin{theorem}
	$(G_0,\nabla, k^+)$ is a local (in fact global) Hessian manifold.
\end{theorem}
\begin{proof}
	A potential function for $k^+$ is 
	$$
	\begin{array}{rl}
		f:&G_0\rightarrow\mathbb{R}\\
		&M\mapsto\frac{1}{2}\mathrm{trace}\left(M^2\right)
	\end{array}
	$$
	In natural coordinates, the gradient has the expression 
	\begin{align*}
		grad\;f&=\sum_{i=1}^4 \frac{\partial f}{\partial x_i}\frac{\partial }{\partial x_i}\\
		&=x_1\frac{\partial }{\partial x_1}+x_3\frac{\partial }{\partial x_2}+x_2\frac{\partial }{\partial x_3}+x_4\frac{\partial }{\partial x_4}
	\end{align*}
	Hence 
	\begin{equation*}
		\left(\frac{\partial^2f}{\partial x_i\partial x_j}\right)=\left(
		\begin{array}{cccc}
			1&0&0&0\\
			0&0&1&0\\
			0&1&0&0\\
			0&0&0&1
		\end{array}
		\right)
	\end{equation*}
\end{proof}
\section{Python  Computes}\label{python}
\begin{tcolorbox}[size=fbox, boxrule=1pt, pad at break*=1mm,colback=cellbackground, colframe=cellborder]
\prompt{In}{incolor}{1}{\boxspacing}
\begin{Verbatim}[commandchars=\\\{\}]
\PY{k+kn}{import} \PY{n+nn}{numpy} \PY{k}{as} \PY{n+nn}{np}
\PY{k+kn}{from} \PY{n+nn}{sympy} \PY{k+kn}{import} \PY{n}{symbols}\PY{p}{,} \PY{n}{Matrix}\PY{p}{,} \PY{n}{simplify}\PY{p}{,} \PY{n}{sqrt}\PY{p}{,} \PY{n}{solve}\PY{p}{,} \PY{n}{Transpose}\PY{p}{,} \PY{n}{Trace}
	\end{Verbatim}
\end{tcolorbox}

\hypertarget{k-symmetric-operators}{%
	\subsection{k-symmetric Operators}\label{k-symmetric-operators}}

\begin{tcolorbox}[size=fbox, boxrule=1pt, pad at break*=1mm,colback=cellbackground, colframe=cellborder]
\prompt{In}{incolor}{2}{\boxspacing}
\begin{Verbatim}[commandchars=\\\{\}]
\PY{c+c1}{\PYZsh{} Defining metric  k = Traza(AB) respect to Orthonormal basis}
\PY{n}{k} \PY{o}{=} \PY{n}{Matrix}\PY{p}{(}\PY{p}{[}\PY{p}{[}\PY{o}{\PYZhy{}}\PY{l+m+mi}{1}\PY{p}{,}\PY{l+m+mi}{0}\PY{p}{,}\PY{l+m+mi}{0}\PY{p}{,}\PY{l+m+mi}{0}\PY{p}{]}\PY{p}{,}\PY{p}{[}\PY{l+m+mi}{0}\PY{p}{,}\PY{l+m+mi}{1}\PY{p}{,}\PY{l+m+mi}{0}\PY{p}{,}\PY{l+m+mi}{0}\PY{p}{]}\PY{p}{,}\PY{p}{[}\PY{l+m+mi}{0}\PY{p}{,}\PY{l+m+mi}{0}\PY{p}{,}\PY{l+m+mi}{1}\PY{p}{,}\PY{l+m+mi}{0}\PY{p}{]}\PY{p}{,}\PY{p}{[}\PY{l+m+mi}{0}\PY{p}{,}\PY{l+m+mi}{0}\PY{p}{,}\PY{l+m+mi}{0}\PY{p}{,}\PY{l+m+mi}{1}\PY{p}{]}\PY{p}{]}\PY{p}{)}
		
\PY{c+c1}{\PYZsh{} Defining variables}
\PY{n}{u11}\PY{p}{,} \PY{n}{u12}\PY{p}{,} \PY{n}{u13}\PY{p}{,} \PY{n}{u14}\PY{p}{,} \PY{n}{u21}\PY{p}{,} \PY{n}{u22}\PY{p}{,} \PY{n}{u23}\PY{p}{,} \PY{n}{u24}\PY{p}{,} \PY{n}{u31}\PY{p}{,} \PY{n}{u32}\PY{p}{,} \PY{n}{u33}\PY{p}{,} \PY{n}{u34}\PY{p}{,} \PY{n}{u41}\PY{p}{,} \PY{n}{u42}\PY{p}{,} \PY{n}{u43}\PY{p}{,} \PY{n}{u44}  \PY{o}{=} \PY{n}{symbols}\PY{p}{(}\PY{l+s+s1}{\PYZsq{}}\PY{l+s+s1}{u11 u12 u13 u14 u21 u22 u23 u24 u31 u32 u33 u34 u41 u42 u43 u44}\PY{l+s+s1}{\PYZsq{}}\PY{p}{)}
		
\PY{c+c1}{\PYZsh{} Matrix representation of an operator in the natural basis of GL\PYZus{}2}
\PY{n}{u}\PY{o}{=}\PY{n}{Matrix}\PY{p}{(}\PY{p}{[}\PY{p}{[}\PY{n}{u11}\PY{p}{,} \PY{n}{u12}\PY{p}{,} \PY{n}{u13}\PY{p}{,} \PY{n}{u14}\PY{p}{]}\PY{p}{,} \PY{p}{[}\PY{n}{u21}\PY{p}{,} \PY{n}{u22}\PY{p}{,} \PY{n}{u23}\PY{p}{,} \PY{n}{u24}\PY{p}{]}\PY{p}{,} \PY{p}{[}\PY{n}{u31}\PY{p}{,} \PY{n}{u32}\PY{p}{,} \PY{n}{u33}\PY{p}{,} \PY{n}{u34}\PY{p}{]}\PY{p}{,} \PY{p}{[}\PY{n}{u41}\PY{p}{,} \PY{n}{u42}\PY{p}{,} \PY{n}{u43}\PY{p}{,} \PY{n}{u44}\PY{p}{]}\PY{p}{]}\PY{p}{)} 
	\end{Verbatim}
\end{tcolorbox}

\begin{tcolorbox}[size=fbox, boxrule=1pt, pad at break*=1mm,colback=cellbackground, colframe=cellborder]
\prompt{In}{incolor}{3}{\boxspacing}
\begin{Verbatim}[commandchars=\\\{\}]
\PY{c+c1}{\PYZsh{} Orthonormal basis respect to metric k=Traza(AB)}
\PY{n}{E1}\PY{o}{=}\PY{n}{Matrix}\PY{p}{(}\PY{p}{[}\PY{p}{[}\PY{l+m+mi}{0}\PY{p}{,}\PY{n}{sqrt}\PY{p}{(}\PY{l+m+mi}{2}\PY{p}{)}\PY{o}{/}\PY{l+m+mi}{2}\PY{p}{]}\PY{p}{,}\PY{p}{[}\PY{o}{\PYZhy{}}\PY{n}{sqrt}\PY{p}{(}\PY{l+m+mi}{2}\PY{p}{)}\PY{o}{/}\PY{l+m+mi}{2}\PY{p}{,}\PY{l+m+mi}{0}\PY{p}{]}\PY{p}{]}\PY{p}{)}  \PY{c+c1}{\PYZsh{}Timelike}
\PY{n}{E2}\PY{o}{=}\PY{n}{Matrix}\PY{p}{(}\PY{p}{[}\PY{p}{[}\PY{l+m+mi}{0}\PY{p}{,}\PY{n}{sqrt}\PY{p}{(}\PY{l+m+mi}{2}\PY{p}{)}\PY{o}{/}\PY{l+m+mi}{2}\PY{p}{]}\PY{p}{,}\PY{p}{[}\PY{n}{sqrt}\PY{p}{(}\PY{l+m+mi}{2}\PY{p}{)}\PY{o}{/}\PY{l+m+mi}{2}\PY{p}{,}\PY{l+m+mi}{0}\PY{p}{]}\PY{p}{]}\PY{p}{)}   \PY{c+c1}{\PYZsh{}Spacelike}
\PY{n}{E3}\PY{o}{=}\PY{n}{Matrix}\PY{p}{(}\PY{p}{[}\PY{p}{[}\PY{n}{sqrt}\PY{p}{(}\PY{l+m+mi}{2}\PY{p}{)}\PY{o}{/}\PY{l+m+mi}{2}\PY{p}{,}\PY{l+m+mi}{0}\PY{p}{]}\PY{p}{,}\PY{p}{[}\PY{l+m+mi}{0}\PY{p}{,}\PY{o}{\PYZhy{}}\PY{n}{sqrt}\PY{p}{(}\PY{l+m+mi}{2}\PY{p}{)}\PY{o}{/}\PY{l+m+mi}{2}\PY{p}{]}\PY{p}{]}\PY{p}{)}  \PY{c+c1}{\PYZsh{}Spacelike}
\PY{n}{E4}\PY{o}{=}\PY{n}{Matrix}\PY{p}{(}\PY{p}{[}\PY{p}{[}\PY{n}{sqrt}\PY{p}{(}\PY{l+m+mi}{2}\PY{p}{)}\PY{o}{/}\PY{l+m+mi}{2}\PY{p}{,}\PY{l+m+mi}{0}\PY{p}{]}\PY{p}{,}\PY{p}{[}\PY{l+m+mi}{0}\PY{p}{,}\PY{n}{sqrt}\PY{p}{(}\PY{l+m+mi}{2}\PY{p}{)}\PY{o}{/}\PY{l+m+mi}{2}\PY{p}{]}\PY{p}{]}\PY{p}{)}   \PY{c+c1}{\PYZsh{}Spacelike}
		
\PY{c+c1}{\PYZsh{}Compute k(u(a),b)=k(a,u(b)) on orthonormal basis}
\PY{n}{KO}\PY{o}{=}\PY{n}{Transpose}\PY{p}{(}\PY{n}{u}\PY{p}{)}\PY{o}{*}\PY{n}{k}\PY{o}{\PYZhy{}}\PY{n}{k}\PY{o}{*}\PY{n}{u}
		
\PY{n}{solve}\PY{p}{(}\PY{p}{[}\PY{n}{KO}\PY{p}{]}\PY{p}{,} \PY{n}{u11}\PY{p}{,} \PY{n}{u12}\PY{p}{,} \PY{n}{u13}\PY{p}{,} \PY{n}{u14}\PY{p}{,} \PY{n}{u21}\PY{p}{,} \PY{n}{u22}\PY{p}{,} \PY{n}{u23}\PY{p}{,} \PY{n}{u24}\PY{p}{,} \PY{n}{u31}\PY{p}{,} \PY{n}{u32}\PY{p}{,} \PY{n}{u33}\PY{p}{,} \PY{n}{u34}\PY{p}{,} \PY{n}{u41}\PY{p}{,} \PY{n}{u42}\PY{p}{,} \PY{n}{u43}\PY{p}{,} \PY{n}{u44}\PY{p}{,} \PY{n+nb}{dict}\PY{o}{=}\PY{k+kc}{True}\PY{p}{)}
	\end{Verbatim}
\end{tcolorbox}

\begin{tcolorbox}[size=fbox, boxrule=.5pt, pad at break*=1mm, opacityfill=0]
\prompt{Out}{outcolor}{3}{\boxspacing}
\begin{Verbatim}[commandchars=\\\{\}]
[\{u12: -u21, u13: -u31, u14: -u41, u23: u32, u24: u42, u34: u43\}]
	\end{Verbatim}
\end{tcolorbox}

\begin{tcolorbox}[size=fbox, boxrule=1pt, pad at break*=1mm,colback=cellbackground, colframe=cellborder]
\prompt{In}{incolor}{4}{\boxspacing}
\begin{Verbatim}[commandchars=\\\{\}]
\PY{c+c1}{\PYZsh{} Matrix representation of the k\PYZhy{}symmetric operator in the orthonormal basis}
\PY{n}{ku}\PY{o}{=}\PY{n}{Matrix}\PY{p}{(}\PY{p}{[}\PY{p}{[}\PY{n}{u11}\PY{p}{,}\PY{n}{u12}\PY{p}{,}\PY{n}{u13}\PY{p}{,}\PY{n}{u14}\PY{p}{]}\PY{p}{,}\PY{p}{[}\PY{o}{\PYZhy{}}\PY{n}{u12}\PY{p}{,}\PY{n}{u22}\PY{p}{,}\PY{n}{u23}\PY{p}{,}\PY{n}{u24}\PY{p}{]}\PY{p}{,} \PY{p}{[}\PY{o}{\PYZhy{}}\PY{n}{u13}\PY{p}{,}\PY{n}{u23}\PY{p}{,}\PY{n}{u33}\PY{p}{,}\PY{n}{u34}\PY{p}{]}\PY{p}{,}\PY{p}{[}\PY{o}{\PYZhy{}}\PY{n}{u14}\PY{p}{,}\PY{n}{u24}\PY{p}{,}\PY{n}{u34}\PY{p}{,}\PY{n}{u44}\PY{p}{]}\PY{p}{]}\PY{p}{)}
\PY{n+nb}{print}\PY{p}{(}\PY{l+s+s1}{\PYZsq{}}\PY{l+s+s1}{k\PYZhy{}symmetric operator in the orthonormal basis:}\PY{l+s+s1}{\PYZsq{}}\PY{p}{)}
\PY{n}{ku}
	\end{Verbatim}
\end{tcolorbox}
\begin{tcolorbox}[size=fbox, boxrule=.5pt, pad at break*=1mm, opacityfill=0]
\begin{Verbatim}[commandchars=\\\{\}]
The k-symmetric operator on the orthonormal basis:
\end{Verbatim}    
	\prompt{Out}{outcolor}{4}{}
$\displaystyle \left[\begin{matrix}u_{11} & u_{12} & u_{13} & u_{14}\\- u_{12} & u_{22} & u_{23} & u_{24}\\- u_{13} & u_{23} & u_{33} & u_{34}\\- u_{14} & u_{24} & u_{34} & u_{44}\end{matrix}\right]$
\end{tcolorbox}

\begin{tcolorbox}[ size=fbox, boxrule=1pt, pad at break*=1mm,colback=cellbackground, colframe=cellborder]
\prompt{In}{incolor}{5}{\boxspacing}
\begin{Verbatim}[commandchars=\\\{\}]
\PY{c+c1}{\PYZsh{} Metrics given by the k\PYZhy{}symmetrics operator in the orthonormal basis:}
\PY{n}{Transpose}\PY{p}{(}\PY{n}{ku}\PY{p}{)}\PY{o}{*}\PY{n}{k}
\end{Verbatim}
\end{tcolorbox}

\begin{tcolorbox}[size=fbox, boxrule=.5pt, pad at break*=1mm, opacityfill=0]            
	\prompt{Out}{outcolor}{5}{}
	
$\displaystyle \left[\begin{matrix}- u_{11} & - u_{12} & - u_{13} & - u_{14}\\- u_{12} & u_{22} & u_{23} & u_{24}\\- u_{13} & u_{23} & u_{33} & u_{34}\\- u_{14} & u_{24} & u_{34} & u_{44}\end{matrix}\right]$
\end{tcolorbox}

\hypertarget{curvature}{%
	\subsection{Curvature}\label{curvature}}

\begin{tcolorbox}[ size=fbox, boxrule=1pt, pad at break*=1mm,colback=cellbackground, colframe=cellborder]
\prompt{In}{incolor}{6}{\boxspacing}
\begin{Verbatim}[commandchars=\\\{\}]
\PY{n}{K0}\PY{o}{=} \PY{n}{Matrix}\PY{p}{(}\PY{p}{[}\PY{p}{[}\PY{l+m+mi}{1}\PY{p}{,}\PY{l+m+mi}{0}\PY{p}{,}\PY{l+m+mi}{0}\PY{p}{,}\PY{l+m+mi}{0}\PY{p}{]}\PY{p}{,}\PY{p}{[}\PY{l+m+mi}{0}\PY{p}{,}\PY{l+m+mi}{1}\PY{p}{,}\PY{l+m+mi}{0}\PY{p}{,}\PY{l+m+mi}{0}\PY{p}{]}\PY{p}{,}\PY{p}{[}\PY{l+m+mi}{0}\PY{p}{,}\PY{l+m+mi}{0}\PY{p}{,}\PY{l+m+mi}{1}\PY{p}{,}\PY{l+m+mi}{0}\PY{p}{]}\PY{p}{,}\PY{p}{[}\PY{l+m+mi}{0}\PY{p}{,}\PY{l+m+mi}{0}\PY{p}{,}\PY{l+m+mi}{0}\PY{p}{,}\PY{l+m+mi}{1}\PY{p}{]}\PY{p}{]}\PY{p}{)}
\PY{c+c1}{\PYZsh{}K1 = k is Levi\PYZhy{}Civita}
\PY{n}{K2}\PY{o}{=} \PY{n}{Matrix}\PY{p}{(}\PY{p}{[}\PY{p}{[}\PY{o}{\PYZhy{}}\PY{l+m+mi}{1}\PY{p}{,}\PY{l+m+mi}{0}\PY{p}{,}\PY{l+m+mi}{0}\PY{p}{,}\PY{l+m+mi}{0}\PY{p}{]}\PY{p}{,}\PY{p}{[}\PY{l+m+mi}{0}\PY{p}{,}\PY{o}{\PYZhy{}}\PY{l+m+mi}{1}\PY{p}{,}\PY{l+m+mi}{0}\PY{p}{,}\PY{l+m+mi}{0}\PY{p}{]}\PY{p}{,}\PY{p}{[}\PY{l+m+mi}{0}\PY{p}{,}\PY{l+m+mi}{0}\PY{p}{,}\PY{l+m+mi}{1}\PY{p}{,}\PY{l+m+mi}{0}\PY{p}{]}\PY{p}{,}\PY{p}{[}\PY{l+m+mi}{0}\PY{p}{,}\PY{l+m+mi}{0}\PY{p}{,}\PY{l+m+mi}{0}\PY{p}{,}\PY{l+m+mi}{1}\PY{p}{]}\PY{p}{]}\PY{p}{)}
	\end{Verbatim}
\end{tcolorbox}

\begin{tcolorbox}[ size=fbox, boxrule=1pt, pad at break*=1mm,colback=cellbackground, colframe=cellborder]
\prompt{In}{incolor}{7}{\boxspacing}
\begin{Verbatim}[commandchars=\\\{\}]
\PY{c+c1}{\PYZsh{} Levi\PYZhy{}Civita Christoffel Symbols}
\PY{n}{gamma} \PY{o}{=} \PY{p}{[}\PY{p}{[}\PY{p}{[}\PY{l+m+mi}{0} \PY{k}{for} \PY{n}{i} \PY{o+ow}{in} \PY{n+nb}{range}\PY{p}{(}\PY{l+m+mi}{4}\PY{p}{)}\PY{p}{]} \PY{k}{for} \PY{n}{j} \PY{o+ow}{in} \PY{n+nb}{range}\PY{p}{(}\PY{l+m+mi}{4}\PY{p}{)}\PY{p}{]} \PY{k}{for} \PY{n}{k} \PY{o+ow}{in} \PY{n+nb}{range}\PY{p}{(}\PY{l+m+mi}{4}\PY{p}{)}\PY{p}{]}
		
\PY{c+c1}{\PYZsh{} Structure coeficients of Lie algebra C\PYZus{}\PYZob{}ij\PYZcb{}\PYZca{}k=AE[i][j][k]}
\PY{n}{AE}\PY{o}{=} \PY{p}{[}\PY{p}{[}\PY{p}{[}\PY{l+m+mi}{0}\PY{p}{,}\PY{l+m+mi}{0}\PY{p}{,}\PY{l+m+mi}{0}\PY{p}{,}\PY{l+m+mi}{0}\PY{p}{]}\PY{p}{,}\PY{p}{[}\PY{l+m+mi}{0}\PY{p}{,}\PY{l+m+mi}{0}\PY{p}{,}\PY{n}{sqrt}\PY{p}{(}\PY{l+m+mi}{2}\PY{p}{)}\PY{p}{,}\PY{l+m+mi}{0}\PY{p}{]}\PY{p}{,}\PY{p}{[}\PY{l+m+mi}{0}\PY{p}{,}\PY{o}{\PYZhy{}}\PY{n}{sqrt}\PY{p}{(}\PY{l+m+mi}{2}\PY{p}{)}\PY{p}{,}\PY{l+m+mi}{0}\PY{p}{,}\PY{l+m+mi}{0}\PY{p}{]}\PY{p}{,}\PY{p}{[}\PY{l+m+mi}{0}\PY{p}{,}\PY{l+m+mi}{0}\PY{p}{,}\PY{l+m+mi}{0}\PY{p}{,}\PY{l+m+mi}{0}\PY{p}{]}\PY{p}{]}\PY{p}{,}
\PY{p}{[}\PY{p}{[}\PY{l+m+mi}{0}\PY{p}{,}\PY{l+m+mi}{0}\PY{p}{,}\PY{o}{\PYZhy{}}\PY{n}{sqrt}\PY{p}{(}\PY{l+m+mi}{2}\PY{p}{)}\PY{p}{,}\PY{l+m+mi}{0}\PY{p}{]}\PY{p}{,}\PY{p}{[}\PY{l+m+mi}{0}\PY{p}{,}\PY{l+m+mi}{0}\PY{p}{,}\PY{l+m+mi}{0}\PY{p}{,}\PY{l+m+mi}{0}\PY{p}{,}\PY{l+m+mi}{0}\PY{p}{,}\PY{l+m+mi}{0}\PY{p}{]}\PY{p}{,}\PY{p}{[}\PY{o}{\PYZhy{}}\PY{n}{sqrt}\PY{p}{(}\PY{l+m+mi}{2}\PY{p}{)}\PY{p}{,}\PY{l+m+mi}{0}\PY{p}{,}\PY{l+m+mi}{0}\PY{p}{,}\PY{l+m+mi}{0}\PY{p}{]}\PY{p}{,}\PY{p}{[}\PY{l+m+mi}{0}\PY{p}{,}\PY{l+m+mi}{0}\PY{p}{,}\PY{l+m+mi}{0}\PY{p}{,}\PY{l+m+mi}{0}\PY{p}{]}\PY{p}{]}\PY{p}{,}
\PY{p}{[}\PY{p}{[}\PY{l+m+mi}{0}\PY{p}{,}\PY{n}{sqrt}\PY{p}{(}\PY{l+m+mi}{2}\PY{p}{)}\PY{p}{,}\PY{l+m+mi}{0}\PY{p}{,}\PY{l+m+mi}{0}\PY{p}{]}\PY{p}{,}\PY{p}{[}\PY{n}{sqrt}\PY{p}{(}\PY{l+m+mi}{2}\PY{p}{)}\PY{p}{,}\PY{l+m+mi}{0}\PY{p}{,}\PY{l+m+mi}{0}\PY{p}{,}\PY{l+m+mi}{0}\PY{p}{]}\PY{p}{,}\PY{p}{[}\PY{l+m+mi}{0}\PY{p}{,}\PY{l+m+mi}{0}\PY{p}{,}\PY{l+m+mi}{0}\PY{p}{,}\PY{l+m+mi}{0}\PY{p}{]}\PY{p}{,}\PY{p}{[}\PY{l+m+mi}{0}\PY{p}{,}\PY{l+m+mi}{0}\PY{p}{,}\PY{l+m+mi}{0}\PY{p}{,}\PY{l+m+mi}{0}\PY{p}{]}\PY{p}{]}\PY{p}{,}
\PY{p}{[}\PY{p}{[}\PY{l+m+mi}{0}\PY{p}{,}\PY{l+m+mi}{0}\PY{p}{,}\PY{l+m+mi}{0}\PY{p}{,}\PY{l+m+mi}{0}\PY{p}{]}\PY{p}{,}\PY{p}{[}\PY{l+m+mi}{0}\PY{p}{,}\PY{l+m+mi}{0}\PY{p}{,}\PY{l+m+mi}{0}\PY{p}{,}\PY{l+m+mi}{0}\PY{p}{]}\PY{p}{,}\PY{p}{[}\PY{l+m+mi}{0}\PY{p}{,}\PY{l+m+mi}{0}\PY{p}{,}\PY{l+m+mi}{0}\PY{p}{,}\PY{l+m+mi}{0}\PY{p}{]}\PY{p}{,}\PY{p}{[}\PY{l+m+mi}{0}\PY{p}{,}\PY{l+m+mi}{0}\PY{p}{,}\PY{l+m+mi}{0}\PY{p}{,}\PY{l+m+mi}{0}\PY{p}{]}\PY{p}{]}\PY{p}{]}
		
\PY{c+c1}{\PYZsh{} Christoffel Symbols Computes }
\PY{c+c1}{\PYZsh{} Gamma\PYZus{}\PYZob{}ij\PYZcb{}\PYZca{}k=\PYZbs{}sum\PYZus{}\PYZob{}l,m,n\PYZcb{}0.5K\PYZca{}\PYZob{}kl\PYZcb{}(\PYZhy{}K\PYZus{}\PYZob{}jm\PYZcb{}C\PYZus{}\PYZob{}il\PYZcb{}\PYZca{}m\PYZhy{}K[l,m]*C\PYZus{}\PYZob{}ji\PYZcb{}\PYZca{}m+K\PYZus{}\PYZob{}im\PYZcb{}*C\PYZus{}\PYZob{}lj\PYZcb{}\PYZca{}m)}
\PY{c+c1}{\PYZsh{} K0}
\PY{k}{for} \PY{n}{i} \PY{o+ow}{in} \PY{n+nb}{range}\PY{p}{(}\PY{l+m+mi}{4}\PY{p}{)}\PY{p}{:}
\PY{k}{for} \PY{n}{j} \PY{o+ow}{in} \PY{n+nb}{range}\PY{p}{(}\PY{l+m+mi}{4}\PY{p}{)}\PY{p}{:}
\PY{k}{for} \PY{n}{k} \PY{o+ow}{in} \PY{n+nb}{range}\PY{p}{(}\PY{l+m+mi}{4}\PY{p}{)}\PY{p}{:}
\PY{n}{gamma}\PY{p}{[}\PY{n}{i}\PY{p}{]}\PY{p}{[}\PY{n}{j}\PY{p}{]}\PY{p}{[}\PY{n}{k}\PY{p}{]}\PY{o}{=}\PY{l+m+mi}{0}
\PY{k}{for} \PY{n}{l} \PY{o+ow}{in} \PY{n+nb}{range}\PY{p}{(}\PY{l+m+mi}{4}\PY{p}{)}\PY{p}{:}
\PY{k}{for} \PY{n}{m} \PY{o+ow}{in} \PY{n+nb}{range}\PY{p}{(}\PY{l+m+mi}{4}\PY{p}{)}\PY{p}{:}
\PY{n}{gamma}\PY{p}{[}\PY{n}{i}\PY{p}{]}\PY{p}{[}\PY{n}{j}\PY{p}{]}\PY{p}{[}\PY{n}{k}\PY{p}{]} \PY{o}{+}\PY{o}{=} \PY{n}{simplify}\PY{p}{(}\PY{l+m+mf}{0.5} \PY{o}{*} \PY{p}{(}\PY{n}{K0}\PY{p}{[}\PY{n}{k}\PY{p}{,}\PY{n}{l}\PY{p}{]}\PY{o}{*}\PY{p}{(}\PY{o}{\PYZhy{}}\PY{n}{K0}\PY{p}{[}\PY{n}{j}\PY{p}{,}\PY{n}{m}\PY{p}{]}\PY{o}{*}\PY{n}{AE}\PY{p}{[}\PY{n}{i}\PY{p}{]}\PY{p}{[}\PY{n}{l}\PY{p}{]}\PY{p}{[}\PY{n}{m}\PY{p}{]}\PY{o}{\PYZhy{}}\PY{n}{K0}\PY{p}{[}\PY{n}{l}\PY{p}{,}\PY{n}{m}\PY{p}{]}\PY{o}{*}\PY{n}{AE}\PY{p}{[}\PY{n}{j}\PY{p}{]}\PY{p}{[}\PY{n}{i}\PY{p}{]}\PY{p}{[}\PY{n}{m}\PY{p}{]}\PY{o}{+}\PY{n}{K0}\PY{p}{[}\PY{n}{i}\PY{p}{,}\PY{n}{m}\PY{p}{]}\PY{o}{*} \PY{n}{AE}\PY{p}{[}\PY{n}{l}\PY{p}{]}\PY{p}{[}\PY{n}{j}\PY{p}{]}\PY{p}{[}\PY{n}{m}\PY{p}{]}\PY{p}{)}\PY{p}{)}\PY{p}{)}
		
		
\PY{n+nb}{print}\PY{p}{(}\PY{l+s+s2}{\PYZdq{}}\PY{l+s+s2}{The Christoffel Symbols are: }\PY{l+s+s2}{\PYZdq{}}\PY{p}{)}
\PY{k}{for} \PY{n}{i} \PY{o+ow}{in} \PY{n+nb}{range}\PY{p}{(}\PY{l+m+mi}{4}\PY{p}{)}\PY{p}{:}
\PY{k}{for} \PY{n}{j} \PY{o+ow}{in} \PY{n+nb}{range}\PY{p}{(}\PY{l+m+mi}{4}\PY{p}{)}\PY{p}{:}
\PY{k}{for} \PY{n}{k} \PY{o+ow}{in} \PY{n+nb}{range}\PY{p}{(}\PY{l+m+mi}{4}\PY{p}{)}\PY{p}{:}
\PY{n+nb}{print}\PY{p}{(}\PY{l+s+s2}{\PYZdq{}}\PY{l+s+s2}{Gamma}\PY{l+s+s2}{\PYZdq{}}\PY{p}{,} \PY{n}{i}\PY{o}{+}\PY{l+m+mi}{1}\PY{p}{,} \PY{n}{j}\PY{o}{+}\PY{l+m+mi}{1}\PY{p}{,} \PY{n}{k}\PY{o}{+}\PY{l+m+mi}{1}\PY{p}{,} \PY{l+s+s2}{\PYZdq{}}\PY{l+s+s2}{=}\PY{l+s+s2}{\PYZdq{}}\PY{p}{,} \PY{n}{gamma}\PY{p}{[}\PY{n}{i}\PY{p}{]}\PY{p}{[}\PY{n}{j}\PY{p}{]}\PY{p}{[}\PY{n}{k}\PY{p}{]}\PY{p}{)}
	\end{Verbatim}
\end{tcolorbox}
\begin{tcolorbox}[size=fbox, boxrule=.5pt, pad at break*=1mm, opacityfill=0]
	\prompt{Out}{outcolor}{7}{}
\end{tcolorbox}
\begin{Verbatim}[commandchars=\\\{\}]
The Christoffel Symbols are:
Gamma 1 1 1 = 0
Gamma 1 1 2 = 0
Gamma 1 1 3 = 0
Gamma 1 1 4 = 0
Gamma 1 2 1 = 0
Gamma 1 2 2 = 0
Gamma 1 2 3 = 1.5*sqrt(2)
Gamma 1 2 4 = 0
Gamma 1 3 1 = 0
Gamma 1 3 2 = -1.5*sqrt(2)
Gamma 1 3 3 = 0
Gamma 1 3 4 = 0
Gamma 1 4 1 = 0
Gamma 1 4 2 = 0
Gamma 1 4 3 = 0
Gamma 1 4 4 = 0
Gamma 2 1 1 = 0
Gamma 2 1 2 = 0
Gamma 2 1 3 = 0.5*sqrt(2)
Gamma 2 1 4 = 0
Gamma 2 2 1 = 0
Gamma 2 2 2 = 0
Gamma 2 2 3 = 0
Gamma 2 2 4 = 0
Gamma 2 3 1 = -0.5*sqrt(2)
Gamma 2 3 2 = 0
Gamma 2 3 3 = 0
Gamma 2 3 4 = 0
Gamma 2 4 1 = 0
Gamma 2 4 2 = 0
Gamma 2 4 3 = 0
Gamma 2 4 4 = 0
Gamma 3 1 1 = 0
Gamma 3 1 2 = -0.5*sqrt(2)
Gamma 3 1 3 = 0
Gamma 3 1 4 = 0
Gamma 3 2 1 = 0.5*sqrt(2)
Gamma 3 2 2 = 0
Gamma 3 2 3 = 0
Gamma 3 2 4 = 0
Gamma 3 3 1 = 0
Gamma 3 3 2 = 0
Gamma 3 3 3 = 0
Gamma 3 3 4 = 0
Gamma 3 4 1 = 0
Gamma 3 4 2 = 0
Gamma 3 4 3 = 0
Gamma 3 4 4 = 0
Gamma 4 1 1 = 0
Gamma 4 1 2 = 0
Gamma 4 1 3 = 0
Gamma 4 1 4 = 0
Gamma 4 2 1 = 0
Gamma 4 2 2 = 0
Gamma 4 2 3 = 0
Gamma 4 2 4 = 0
Gamma 4 3 1 = 0
Gamma 4 3 2 = 0
Gamma 4 3 3 = 0
Gamma 4 3 4 = 0
Gamma 4 4 1 = 0
Gamma 4 4 2 = 0
Gamma 4 4 3 = 0
Gamma 4 4 4 = 0
\end{Verbatim}

\begin{tcolorbox}[ size=fbox, boxrule=1pt, pad at break*=1mm,colback=cellbackground, colframe=cellborder]
\prompt{In}{incolor}{8}{\boxspacing}
\begin{Verbatim}[commandchars=\\\{\}]
\PY{c+c1}{\PYZsh{} Curvature of K\PYZus{}0}
\PY{n}{L1}\PY{o}{=}\PY{n}{Matrix}\PY{p}{(}\PY{p}{[}\PY{p}{[}\PY{l+m+mi}{0}\PY{p}{,}\PY{l+m+mi}{0}\PY{p}{,}\PY{l+m+mi}{0}\PY{p}{,}\PY{l+m+mi}{0}\PY{p}{]}\PY{p}{,}\PY{p}{[}\PY{l+m+mi}{0}\PY{p}{,}\PY{l+m+mi}{0}\PY{p}{,}\PY{o}{\PYZhy{}}\PY{l+m+mi}{3}\PY{o}{*}\PY{n}{sqrt}\PY{p}{(}\PY{l+m+mi}{2}\PY{p}{)}\PY{o}{/}\PY{l+m+mi}{2}\PY{p}{,}\PY{l+m+mi}{0}\PY{p}{]}\PY{p}{,}\PY{p}{[}\PY{l+m+mi}{0}\PY{p}{,}\PY{l+m+mi}{3}\PY{o}{*}\PY{n}{sqrt}\PY{p}{(}\PY{l+m+mi}{2}\PY{p}{)}\PY{o}{/}\PY{l+m+mi}{2}\PY{p}{,}\PY{l+m+mi}{0}\PY{p}{,}\PY{l+m+mi}{0}\PY{p}{]}\PY{p}{,}\PY{p}{[}\PY{l+m+mi}{0}\PY{p}{,}\PY{l+m+mi}{0}\PY{p}{,}\PY{l+m+mi}{0}\PY{p}{,}\PY{l+m+mi}{0}\PY{p}{]}\PY{p}{]}\PY{p}{)}
\PY{n}{L2}\PY{o}{=}\PY{n}{Matrix}\PY{p}{(}\PY{p}{[}\PY{p}{[}\PY{l+m+mi}{0}\PY{p}{,}\PY{l+m+mi}{0}\PY{p}{,}\PY{o}{\PYZhy{}}\PY{n}{sqrt}\PY{p}{(}\PY{l+m+mi}{2}\PY{p}{)}\PY{o}{/}\PY{l+m+mi}{2}\PY{p}{,}\PY{l+m+mi}{0}\PY{p}{]}\PY{p}{,}\PY{p}{[}\PY{l+m+mi}{0}\PY{p}{,}\PY{l+m+mi}{0}\PY{p}{,}\PY{l+m+mi}{0}\PY{p}{,}\PY{l+m+mi}{0}\PY{p}{]}\PY{p}{,}\PY{p}{[}\PY{n}{sqrt}\PY{p}{(}\PY{l+m+mi}{2}\PY{p}{)}\PY{o}{/}\PY{l+m+mi}{2}\PY{p}{,}\PY{l+m+mi}{0}\PY{p}{,}\PY{l+m+mi}{0}\PY{p}{,}\PY{l+m+mi}{0}\PY{p}{]}\PY{p}{,}\PY{p}{[}\PY{l+m+mi}{0}\PY{p}{,}\PY{l+m+mi}{0}\PY{p}{,}\PY{l+m+mi}{0}\PY{p}{,}\PY{l+m+mi}{0}\PY{p}{]}\PY{p}{]}\PY{p}{)}
\PY{n}{L3}\PY{o}{=}\PY{n}{Matrix}\PY{p}{(}\PY{p}{[}\PY{p}{[}\PY{l+m+mi}{0}\PY{p}{,}\PY{n}{sqrt}\PY{p}{(}\PY{l+m+mi}{2}\PY{p}{)}\PY{o}{/}\PY{l+m+mi}{2}\PY{p}{,}\PY{l+m+mi}{0}\PY{p}{,}\PY{l+m+mi}{0}\PY{p}{]}\PY{p}{,}\PY{p}{[}\PY{o}{\PYZhy{}}\PY{n}{sqrt}\PY{p}{(}\PY{l+m+mi}{2}\PY{p}{)}\PY{o}{/}\PY{l+m+mi}{2}\PY{p}{,}\PY{l+m+mi}{0}\PY{p}{,}\PY{l+m+mi}{0}\PY{p}{,}\PY{l+m+mi}{0}\PY{p}{]}\PY{p}{,}\PY{p}{[}\PY{l+m+mi}{0}\PY{p}{,}\PY{l+m+mi}{0}\PY{p}{,}\PY{l+m+mi}{0}\PY{p}{,}\PY{l+m+mi}{0}\PY{p}{]}\PY{p}{,}\PY{p}{[}\PY{l+m+mi}{0}\PY{p}{,}\PY{l+m+mi}{0}\PY{p}{,}\PY{l+m+mi}{0}\PY{p}{,}\PY{l+m+mi}{0}\PY{p}{]}\PY{p}{]}\PY{p}{)}
		
\PY{c+c1}{\PYZsh{} R(e\PYZus{}1,e\PYZus{}3, )\PYZbs{}neq0}
\PY{n}{L3}\PY{o}{*}\PY{n}{L1}\PY{o}{\PYZhy{}}\PY{n}{L1}\PY{o}{*}\PY{n}{L3}\PY{o}{+}\PY{n}{sqrt}\PY{p}{(}\PY{l+m+mi}{2}\PY{p}{)}\PY{o}{*}\PY{n}{L2}
	\end{Verbatim}
\end{tcolorbox}

\begin{tcolorbox}[size=fbox, boxrule=.5pt, pad at break*=1mm, opacityfill=0]          
\prompt{Out}{outcolor}{8}{}
	
$\displaystyle \left[\begin{matrix}0 & 0 & - \frac{5}{2} & 0\\0 & 0 & 0 & 0\\\frac{5}{2} & 0 & 0 & 0\\0 & 0 & 0 & 0\end{matrix}\right]$
\end{tcolorbox}

\begin{tcolorbox}[ size=fbox, boxrule=1pt, pad at break*=1mm,colback=cellbackground, colframe=cellborder]
\prompt{In}{incolor}{9}{\boxspacing}
\begin{Verbatim}[commandchars=\\\{\}]
\PY{c+c1}{\PYZsh{} K2}
\PY{k}{for} \PY{n}{i} \PY{o+ow}{in} \PY{n+nb}{range}\PY{p}{(}\PY{l+m+mi}{4}\PY{p}{)}\PY{p}{:}
\PY{k}{for} \PY{n}{j} \PY{o+ow}{in} \PY{n+nb}{range}\PY{p}{(}\PY{l+m+mi}{4}\PY{p}{)}\PY{p}{:}
\PY{k}{for} \PY{n}{k} \PY{o+ow}{in} \PY{n+nb}{range}\PY{p}{(}\PY{l+m+mi}{4}\PY{p}{)}\PY{p}{:}
\PY{n}{gamma}\PY{p}{[}\PY{n}{i}\PY{p}{]}\PY{p}{[}\PY{n}{j}\PY{p}{]}\PY{p}{[}\PY{n}{k}\PY{p}{]}\PY{o}{=}\PY{l+m+mi}{0}
\PY{k}{for} \PY{n}{l} \PY{o+ow}{in} \PY{n+nb}{range}\PY{p}{(}\PY{l+m+mi}{4}\PY{p}{)}\PY{p}{:}
\PY{k}{for} \PY{n}{m} \PY{o+ow}{in} \PY{n+nb}{range}\PY{p}{(}\PY{l+m+mi}{4}\PY{p}{)}\PY{p}{:}
\PY{n}{gamma}\PY{p}{[}\PY{n}{i}\PY{p}{]}\PY{p}{[}\PY{n}{j}\PY{p}{]}\PY{p}{[}\PY{n}{k}\PY{p}{]} \PY{o}{+}\PY{o}{=} \PY{n}{simplify}\PY{p}{(}\PY{l+m+mf}{0.5} \PY{o}{*} \PY{p}{(}\PY{n}{K2}\PY{p}{[}\PY{n}{k}\PY{p}{,}\PY{n}{l}\PY{p}{]}\PY{o}{*}\PY{p}{(}\PY{o}{\PYZhy{}}\PY{n}{K2}\PY{p}{[}\PY{n}{j}\PY{p}{,}\PY{n}{m}\PY{p}{]}\PY{o}{*}\PY{n}{AE}\PY{p}{[}\PY{n}{i}\PY{p}{]}\PY{p}{[}\PY{n}{l}\PY{p}{]}\PY{p}{[}\PY{n}{m}\PY{p}{]}\PY{o}{\PYZhy{}}\PY{n}{K2}\PY{p}{[}\PY{n}{l}\PY{p}{,}\PY{n}{m}\PY{p}{]}\PY{o}{*}\PY{n}{AE}\PY{p}{[}\PY{n}{j}\PY{p}{]}\PY{p}{[}\PY{n}{i}\PY{p}{]}\PY{p}{[}\PY{n}{m}\PY{p}{]}\PY{o}{+}\PY{n}{K2}\PY{p}{[}\PY{n}{i}\PY{p}{,}\PY{n}{m}\PY{p}{]}\PY{o}{*} \PY{n}{AE}\PY{p}{[}\PY{n}{l}\PY{p}{]}\PY{p}{[}\PY{n}{j}\PY{p}{]}\PY{p}{[}\PY{n}{m}\PY{p}{]}\PY{p}{)}\PY{p}{)}\PY{p}{)}
		
		
\PY{n+nb}{print}\PY{p}{(}\PY{l+s+s2}{\PYZdq{}}\PY{l+s+s2}{The Christoffel Symbols are: }\PY{l+s+s2}{\PYZdq{}}\PY{p}{)}
\PY{k}{for} \PY{n}{i} \PY{o+ow}{in} \PY{n+nb}{range}\PY{p}{(}\PY{l+m+mi}{4}\PY{p}{)}\PY{p}{:}
\PY{k}{for} \PY{n}{j} \PY{o+ow}{in} \PY{n+nb}{range}\PY{p}{(}\PY{l+m+mi}{4}\PY{p}{)}\PY{p}{:}
\PY{k}{for} \PY{n}{k} \PY{o+ow}{in} \PY{n+nb}{range}\PY{p}{(}\PY{l+m+mi}{4}\PY{p}{)}\PY{p}{:}
\PY{n+nb}{print}\PY{p}{(}\PY{l+s+s2}{\PYZdq{}}\PY{l+s+s2}{Gamma}\PY{l+s+s2}{\PYZdq{}}\PY{p}{,} \PY{n}{i}\PY{o}{+}\PY{l+m+mi}{1}\PY{p}{,} \PY{n}{j}\PY{o}{+}\PY{l+m+mi}{1}\PY{p}{,} \PY{n}{k}\PY{o}{+}\PY{l+m+mi}{1}\PY{p}{,} \PY{l+s+s2}{\PYZdq{}}\PY{l+s+s2}{=}\PY{l+s+s2}{\PYZdq{}}\PY{p}{,} \PY{n}{gamma}\PY{p}{[}\PY{n}{i}\PY{p}{]}\PY{p}{[}\PY{n}{j}\PY{p}{]}\PY{p}{[}\PY{n}{k}\PY{p}{]}\PY{p}{)}
	\end{Verbatim}
\end{tcolorbox}

\begin{tcolorbox}[size=fbox, boxrule=.5pt, pad at break*=1mm, opacityfill=0]
	\prompt{Out}{outcolor}{9}{}
\end{tcolorbox}
\begin{Verbatim}[commandchars=\\\{\}]
	The Christoffel Symbols are:
	Gamma 1 1 1 = 0
	Gamma 1 1 2 = 0
	Gamma 1 1 3 = 0
	Gamma 1 1 4 = 0
	Gamma 1 2 1 = 0
	Gamma 1 2 2 = 0
	Gamma 1 2 3 = -0.5*sqrt(2)
	Gamma 1 2 4 = 0
	Gamma 1 3 1 = 0
	Gamma 1 3 2 = -0.5*sqrt(2)
	Gamma 1 3 3 = 0
	Gamma 1 3 4 = 0
	Gamma 1 4 1 = 0
	Gamma 1 4 2 = 0
	Gamma 1 4 3 = 0
	Gamma 1 4 4 = 0
	Gamma 2 1 1 = 0
	Gamma 2 1 2 = 0
	Gamma 2 1 3 = -1.5*sqrt(2)
	Gamma 2 1 4 = 0
	Gamma 2 2 1 = 0
	Gamma 2 2 2 = 0
	Gamma 2 2 3 = 0
	Gamma 2 2 4 = 0
	Gamma 2 3 1 = -1.5*sqrt(2)
	Gamma 2 3 2 = 0
	Gamma 2 3 3 = 0
	Gamma 2 3 4 = 0
	Gamma 2 4 1 = 0
	Gamma 2 4 2 = 0
	Gamma 2 4 3 = 0
	Gamma 2 4 4 = 0
	Gamma 3 1 1 = 0
	Gamma 3 1 2 = 0.5*sqrt(2)
	Gamma 3 1 3 = 0
	Gamma 3 1 4 = 0
	Gamma 3 2 1 = -0.5*sqrt(2)
	Gamma 3 2 2 = 0
	Gamma 3 2 3 = 0
	Gamma 3 2 4 = 0
	Gamma 3 3 1 = 0
	Gamma 3 3 2 = 0
	Gamma 3 3 3 = 0
	Gamma 3 3 4 = 0
	Gamma 3 4 1 = 0
	Gamma 3 4 2 = 0
	Gamma 3 4 3 = 0
	Gamma 3 4 4 = 0
	Gamma 4 1 1 = 0
	Gamma 4 1 2 = 0
	Gamma 4 1 3 = 0
	Gamma 4 1 4 = 0
	Gamma 4 2 1 = 0
	Gamma 4 2 2 = 0
	Gamma 4 2 3 = 0
	Gamma 4 2 4 = 0
	Gamma 4 3 1 = 0
	Gamma 4 3 2 = 0
	Gamma 4 3 3 = 0
	Gamma 4 3 4 = 0
	Gamma 4 4 1 = 0
	Gamma 4 4 2 = 0
	Gamma 4 4 3 = 0
	Gamma 4 4 4 = 0
\end{Verbatim}

\begin{tcolorbox}[ size=fbox, boxrule=1pt, pad at break*=1mm,colback=cellbackground, colframe=cellborder]
\prompt{In}{incolor}{10}{\boxspacing}
\begin{Verbatim}[commandchars=\\\{\}]
\PY{c+c1}{\PYZsh{} Curvature of K\PYZus{}2}
\PY{n}{L21}\PY{o}{=}\PY{n}{Matrix}\PY{p}{(}\PY{p}{[}\PY{p}{[}\PY{l+m+mi}{0}\PY{p}{,}\PY{l+m+mi}{0}\PY{p}{,}\PY{l+m+mi}{0}\PY{p}{,}\PY{l+m+mi}{0}\PY{p}{]}\PY{p}{,}\PY{p}{[}\PY{l+m+mi}{0}\PY{p}{,}\PY{l+m+mi}{0}\PY{p}{,}\PY{o}{\PYZhy{}}\PY{n}{sqrt}\PY{p}{(}\PY{l+m+mi}{2}\PY{p}{)}\PY{o}{/}\PY{l+m+mi}{2}\PY{p}{,}\PY{l+m+mi}{0}\PY{p}{]}\PY{p}{,}\PY{p}{[}\PY{l+m+mi}{0}\PY{p}{,}\PY{o}{\PYZhy{}}\PY{n}{sqrt}\PY{p}{(}\PY{l+m+mi}{2}\PY{p}{)}\PY{o}{/}\PY{l+m+mi}{2}\PY{p}{,}\PY{l+m+mi}{0}\PY{p}{,}\PY{l+m+mi}{0}\PY{p}{]}\PY{p}{,}\PY{p}{[}\PY{l+m+mi}{0}\PY{p}{,}\PY{l+m+mi}{0}\PY{p}{,}\PY{l+m+mi}{0}\PY{p}{,}\PY{l+m+mi}{0}\PY{p}{]}\PY{p}{]}\PY{p}{)}
\PY{n}{L22}\PY{o}{=}\PY{n}{Matrix}\PY{p}{(}\PY{p}{[}\PY{p}{[}\PY{l+m+mi}{0}\PY{p}{,}\PY{l+m+mi}{0}\PY{p}{,}\PY{o}{\PYZhy{}}\PY{l+m+mi}{3}\PY{o}{*}\PY{n}{sqrt}\PY{p}{(}\PY{l+m+mi}{2}\PY{p}{)}\PY{o}{/}\PY{l+m+mi}{2}\PY{p}{,}\PY{l+m+mi}{0}\PY{p}{]}\PY{p}{,}\PY{p}{[}\PY{l+m+mi}{0}\PY{p}{,}\PY{l+m+mi}{0}\PY{p}{,}\PY{l+m+mi}{0}\PY{p}{,}\PY{l+m+mi}{0}\PY{p}{]}\PY{p}{,}\PY{p}{[}\PY{o}{\PYZhy{}}\PY{l+m+mi}{3}\PY{o}{*}\PY{n}{sqrt}\PY{p}{(}\PY{l+m+mi}{2}\PY{p}{)}\PY{o}{/}\PY{l+m+mi}{2}\PY{p}{,}\PY{l+m+mi}{0}\PY{p}{,}\PY{l+m+mi}{0}\PY{p}{,}\PY{l+m+mi}{0}\PY{p}{]}\PY{p}{,}\PY{p}{[}\PY{l+m+mi}{0}\PY{p}{,}\PY{l+m+mi}{0}\PY{p}{,}\PY{l+m+mi}{0}\PY{p}{,}\PY{l+m+mi}{0}\PY{p}{]}\PY{p}{]}\PY{p}{)}
\PY{n}{L23}\PY{o}{=}\PY{n}{Matrix}\PY{p}{(}\PY{p}{[}\PY{p}{[}\PY{l+m+mi}{0}\PY{p}{,}\PY{o}{\PYZhy{}}\PY{n}{sqrt}\PY{p}{(}\PY{l+m+mi}{2}\PY{p}{)}\PY{o}{/}\PY{l+m+mi}{2}\PY{p}{,}\PY{l+m+mi}{0}\PY{p}{,}\PY{l+m+mi}{0}\PY{p}{]}\PY{p}{,}\PY{p}{[}\PY{n}{sqrt}\PY{p}{(}\PY{l+m+mi}{2}\PY{p}{)}\PY{o}{/}\PY{l+m+mi}{2}\PY{p}{,}\PY{l+m+mi}{0}\PY{p}{,}\PY{l+m+mi}{0}\PY{p}{,}\PY{l+m+mi}{0}\PY{p}{]}\PY{p}{,}\PY{p}{[}\PY{l+m+mi}{0}\PY{p}{,}\PY{l+m+mi}{0}\PY{p}{,}\PY{l+m+mi}{0}\PY{p}{,}\PY{l+m+mi}{0}\PY{p}{]}\PY{p}{,}\PY{p}{[}\PY{l+m+mi}{0}\PY{p}{,}\PY{l+m+mi}{0}\PY{p}{,}\PY{l+m+mi}{0}\PY{p}{,}\PY{l+m+mi}{0}\PY{p}{]}\PY{p}{]}\PY{p}{)}
\PY{c+c1}{\PYZsh{} R(e\PYZus{}1,e\PYZus{}2, )\PYZbs{}neq0}
\PY{n}{L21}\PY{o}{*}\PY{n}{L22}\PY{o}{\PYZhy{}}\PY{n}{L22}\PY{o}{*}\PY{n}{L21}\PY{o}{\PYZhy{}}\PY{n}{sqrt}\PY{p}{(}\PY{l+m+mi}{2}\PY{p}{)}\PY{o}{*}\PY{n}{L23}
	\end{Verbatim}
\end{tcolorbox}

\begin{tcolorbox}[size=fbox, boxrule=.5pt, pad at break*=1mm, opacityfill=0]               
	\prompt{Out}{outcolor}{10}{}
	
	$\displaystyle \left[\begin{matrix}0 & - \frac{1}{2} & 0 & 0\\\frac{1}{2} & 0 & 0 & 0\\0 & 0 & 0 & 0\\0 & 0 & 0 & 0\end{matrix}\right]$
\end{tcolorbox}



\end{document}